\documentclass{article}

\usepackage[utf8]{inputenc}
\usepackage{fullpage}
\usepackage{listings}
\usepackage{amsmath}
\usepackage{amsthm}
\usepackage{amssymb}
\usepackage[normalem]{ulem}
\usepackage{enumerate, fullpage}
\usepackage{color}
\usepackage{graphicx}
\usepackage[dvipsnames]{xcolor}
\usepackage{verbatim}
\usepackage{epstopdf}
\usepackage[toc,page]{appendix}
\usepackage{epsfig} 
\usepackage{tablefootnote}
\usepackage{tabularx,booktabs}
\usepackage{multirow}
\usepackage{hyperref}
\usepackage{algorithm2e}
\usepackage{algorithmic}
\usepackage{soul}
\RestyleAlgo{ruled}
\usepackage{graphbox}
\usepackage[square,sort,comma,numbers]{natbib}

\usepackage{epstopdf}
\newtheorem{theorem}{Theorem}
\newtheorem{lemma}{Lemma}
\usepackage{tikz}
\usetikzlibrary{shapes,positioning,fit,arrows}
\newcommand{\cb}{\mathbf{c}}
\newcommand{\cF}{\mathcal{F}}
\newcommand{\cS}{\mathcal{S}}

\newcommand{\commentout}[1]{}
\newcommand{\bF}{\boldsymbol{F}}
\newcommand{\bW}{\boldsymbol{W}}

\newcommand{\bc}{\boldsymbol{c}}
\newcommand{\bb}{\boldsymbol{b}}
\newcommand{\be}{\boldsymbol{e}}

\newcommand{\bx}{\boldsymbol{x}}
\newcommand{\mbNt}{\mathbb{N}_t}
\newcommand{\mbNx}{\mathbb{N}_x}
\DeclareMathOperator*{\argmin}{arg\,min}

\usepackage{multirow, tabularx}
\newcolumntype{C}{>{\centering\arraybackslash}X}
\newcolumntype{s}{>{\hsize=.4\hsize}X}
\newcolumntype{o}{>{\hsize=.1\hsize}X}

\tikzstyle{startstop} = [rectangle, rounded corners, 
minimum width=2cm, 
minimum height=1cm,
text centered, 
text width=12em,
draw=black, 
fill=blue!20]
\tikzstyle{line} = [draw, -latex']
\tikzstyle{block} = [rectangle, rounded corners, 
minimum width=3cm, 
minimum height=1cm,
text centered, 
text width=15em,
draw=black, 
fill=blue!20]
\tikzstyle{freeblock} = [rectangle, rounded corners, 
minimum height=1cm,
text centered, 
draw=black, 
fill=blue!20]
\tikzstyle{io} = [trapezium, 
trapezium stretches=true, 
trapezium left angle=70, 
trapezium right angle=110, 
minimum width=3cm, 
minimum height=1cm, text centered, 
draw=black, fill=blue!30]
\tikzstyle{smallblockpurple} = [rectangle, rounded corners,
minimum width=1cm, 
text width=2cm, 
minimum height=1cm, 
text centered, 
draw=black, 
fill=blue!20]
\tikzstyle{medianblockpurple} = [rectangle, rounded corners,
minimum width=1cm, 
text width=3cm, 
minimum height=1cm, 
text centered, 
draw=black, 
fill=blue!20]
\tikzstyle{smallblock} = [rectangle, rounded corners,
minimum width=1cm, 
minimum height=1cm, 
text centered, 
draw=black, 
fill=orange!20]
\tikzstyle{smallblock_special} = [rectangle, rounded corners,
minimum width=1cm, 
minimum height=1cm, 
text centered, 
draw=black, 
fill=red!20]
\tikzstyle{smallblock_final} = [rectangle, rounded corners,
minimum width=1cm, 
minimum height=1cm, 
text centered, 
draw=black, 
fill=yellow!20]
\tikzstyle{decision} = [diamond, 
text centered, 
draw=black, 
fill=green!30]
\tikzstyle{arrow} = [thick,->,>=stealth]

\title{Fourier Features for Identifying Differential Equations
(FourierIdent)}

\author{Mengyi Tang\thanks{School of Mathematics, Georgia Institute of Technology, Atlanta, GA 30332-0160. Email: tangmengyi@gatech.edu},  
Hao Liu\thanks{
Department of Mathematics, Hong Kong Baptist University. Email: haoliu@ hkbu.edu.hk. Research is partially supported
by HKBU 179356, NSFC 12201530 and HKRGC ECS 22302123.  
}, 
Wenjing Liao\thanks{School of Mathematics, Georgia Institute of Technology, Atlanta, GA 30332-0160. Email: wliao60@gatech.edu. This work was partially funded by NSF 2145167.}, 
Sung Ha Kang\thanks{School of Mathematics, Georgia Institute of Technology, Atlanta, GA 30332-0160. Email: kang@math.gatech.edu. This work was partially funded by Simons Foundation grant 584960.
}
}

\begin{document}

\maketitle

\begin{abstract}  
We investigate the benefits and challenges of utilizing the frequency information in differential equation identification.
Solving  differential equations and Fourier analysis are closely related, yet there is limited work in exploring this
connection in the identification of differential equations.  
Given a single realization of the differential equation perturbed by noise, we aim to identify the underlying differential equation  governed by a linear combination of  linear and nonlinear differential and polynomial terms in the frequency domain.
This is  challenging due to large magnitudes and sensitivity to noise. We introduce a Fourier feature denoising, and define the meaningful data region and the core regions of features to reduce the effect of noise in the frequency domain.  We use Subspace Pursuit on the core region of the time derivative feature, and introduce a group trimming step to refine the support.   We further introduce a new energy based on the core regions of features for coefficient identification.  
Utilizing the core regions of features serves two critical purposes: eliminating the low-response regions dominated by noise, and enhancing the accuracy in coefficient identification. 
The proposed method is tested on various differential equations with linear, nonlinear, and high-order derivative feature terms.  Our results demonstrate the advantages of the proposed method, particularly on complex and highly corrupted datasets.
\end{abstract}

\section{Introduction}
Differential Equations are widely used to describe physical laws. The original discoveries of differential equations  for real-world physical processes typically require
a good understanding of the physical laws, and supportive evidence from empirical observations. In recent years, data-driven discoveries of differential equations become popular, where one can utilize data to identify the underlying equation.  

This paper studies this inverse problem of identifying a differential equation in the form of 
\begin{equation}
  \frac{\partial  u(x,t) }{\partial t} =  G(u)  =  \sum_{l=1}^L c_l g_l(u), 
\label{unknown eqn}
\end{equation}
from a single trajectory of noisy observations. Our objective is to find the unknown coefficient vector $\bc = (c_l)_{l=1,\ldots,L} \in \mathbb{R}^L$, including the support of the vector and the value of each $c_l$ to identify the  differential equation. 
The form of (\ref{unknown eqn}) includes a wide range of PDEs in science and engineering applications, such as the heat equation, transport equation, Burgers' equation, Korteweg-de Vires (KdV) equation, Kuramoto-Sivashinsky (KS) equation, and many others. This inverse problem is often regarded as the model identification problem, where one identifies the active features in a large dictionary, which is more challenging than the parameter estimation problem when the PDE form is given.

The identification of differential equations has a long history in science and engineering, dating back to 1970s \cite{akaike1974new,bellman1969new}. From 1980s to 2000s, parameter estimation in differential equations was widely studied \cite{baake1992fitting,bongard2007automated,muller2004parameter,schmidt2009distilling,he2022numerical}. In recent years, sparse regression was introduced to the model identification problem, which allows one to identify a small number of active features in the underlying equation from a large dictionary, e.g., 
Sparse Identification of Nonlinear Dynamics (SINDy) \cite{brunton2016discovering,fasel2022ensemble, kaheman2020sindy,rudy2019data}, 
PDE identification via data discovery and sparse optimization \cite{schaeffer2017learning},   
using differential terms for the dictionary \cite{he2022robust, kang2021ident,rudy2019data,schaeffer2017learning,he2023group}, and theoretical studies \cite{he2022asymptotic,he2023much}. 
Recent progress using weak formulation of features have significantly improved the robustness in the identification of differential equations, e.g., \cite{golden2023physically,  gurevich2021learning,gurevich2019robust,messenger2021weak, messenger2021weak2, messenger2023coarse,reinbold2020using,tang2023weakident}: A well-designed test function is introduced in  \cite{messenger2021weak, messenger2021weak2}, a statistical approach with bootstrap in \cite{fasel2022ensemble}, and improvements in sparsity-promotion methods  in \cite{he2022robust,russo2022convergence,tang2023weakident}.

In this paper, the governing function $ G(u) $ in \eqref{unknown eqn} is assumed to be a linear combination of linear and nonlinear features of $u$, e.g. the $g_l (u)$'s are either monomials $u^\beta$ or spatial derivatives of monomials  $\frac{\partial^{\alpha}}{\partial x^{\alpha}}(u^\beta)$, where $\alpha,\beta$ are nonnegative integers. Motivated by the weak formulation of features in \cite{messenger2021weak, tang2023weakident}, we consider identifying differential equation directly in the frequency domain:
\begin{equation}
\mathcal{F} \left\{ \frac{\partial  u(x,t) }{\partial t} \right\}= \mathcal{F}\{ G(u) \} =  \sum_{l=1}^L c_l \mathcal{F}\left\{ g_l(u) \right\},
\label{unknown eqn fourier}
\end{equation}
where $\cF = \cF_{x,t}$ denotes the two-dimensional \textit{Fourier transform} in $x$ and $t$.

Parameter estimation for differential equations in the frequency domain has been considered in science and engineering applications.
In \cite{van2014frequency}, a sample maximum likelihood estimator in the frequency domain was used to identify some spatially dependent parameters in a parabolic equation, from the given PDE form. 
In \cite{goos2017frequency}, a frequency domain identification technique  was proposed to estimate Linear Parameter-Varying differential equations with weighted nonlinear least squares.
\cite{jagtap2020adaptive} examined how physics-informed neural networks successively capture  different frequencies of the solution such that the low-frequency components is captured at the beginning of training then the high-frequency components are captured as the training process proceeds. 
The authors in \cite{zhang2021robust} considered cutting the given data in the frequency domain after numerical differentiation of the neural-network-denoised data in the model identification problem.
The authors in \cite{zhao2022much} provided a mathematical theory on the possibility of learning a PDE from a single solution trajectory. 
While the frequency information of PDEs has  been considered in some identification methods in the literature, there is no general framework to identify an unknown PDE using frequency responses.

We propose to use Fourier features for IDENTifying differential equations  (FourierIdent) in this paper. After taking the Fourier transform of \eqref{unknown eqn}, we perform model selection and parameter estimation on \eqref{unknown eqn fourier} in the frequency domain, which is different from existing works on model identification in the physical domain.  For example in \cite{messenger2021weak, tang2023weakident}, the integral forms are calculated through convolution and evaluated in the frequency domain using the Fast Fourier Transform (FFT), while the rest of the computation is performed in the physical domain with the inverse FFT.   
We propose a robust framework, FourierIdent, for identifying differential equations directly in the frequency domain. If the coefficients are successfully identified, the resulting governing equations in the physical domain and in the frequency domain are identical.

The contributions of this paper are summarized as follows:
\begin{enumerate}
    \item We propose a stable denoising method in the frequency domain to  {effectively} handle  large frequency magnitudes   and to improve the robustness against noise.  We use Fourier analysis and fit the correct decay of the coefficients to define the meangingful data region, and define the core regions of features for identification.  The Fourier transform is computed through FFT for  efficiency.  
    \item We develop FourierIdent to identify the differential equation from a single noisy trajectory of the PDE in frequency domain.   Within this framework, we use the Subspace Pursuit greedy algorithm and propose group trimming to find the coefficient support at each sparsity level.  We further propose a new energy based on the core regions of features to refine the coefficient identification. 
    \item FourierIdent shows benefits in identifying differential equations under high noise levels, and when 
    the given data have a complex pattern, e.g., with rich information in different frequency modes.  We present  various numerical experimental results in comparison with WSINDy \cite{messenger2021weak} and WeakIdent \cite{tang2023weakident}. 
\end{enumerate}

This paper is organized as follows: In Section \ref{s: problem description}, we provide the problem setup and the formulation of Fourier features. The error analysis of using one Fourier feature is discussed in Subsection \ref{ss: error analysis of fourier features}. In Section \ref{s: denoising a feature}, we propose the denoising method for Fourier features and define the core regions of features in the frequency domain.  
In Section \ref{s: feature selection and coefficient recovery}, we present the methodology of FourerIdent,  using Subspace Pursuit, group trimming and the new enegry based on the core regions of features for coefficient identification.  
In Section \ref{s: numerical implementation details} and \ref{s: numerical experiments}, we provide details of implementation and numerical results of FourierIdent. 
We conclude the paper with remarks in Section \ref{s:conclusion}.

\section{Problem set-up and Fourier features} \label{s: problem description}
We present FourierIdent on one-dimensional PDEs, but our proposed method can be  extended to high-dimensional PDEs.
We let the physical domain be  $\Omega = [0, X] \times [0, T]$ with $X>0, T>0$, and data to be sampled on a uniform grid with step size $\Delta x$, and $\Delta t$. 
  We denote  $x_i = i \Delta x, t^n = n \Delta t$, and the PDE solution at $(x_i,t^n)$ as  $U_i^n=u(x_i,t^n)$. Our observation of the solution at $(x_i,t^n)$ is contaminated by noise such that
\begin{equation}
    U_{noise, i}^n = U_i^n + \epsilon_i^n \;\; {\rm for}\;\;(x_i,t^n) \in \Omega,
    \label{e:u+eps}
\end{equation}
where $\epsilon_i^n$ represents zero-mean noise at each point $(x_i, t^n)$. 
We denote the given data as 
\begin{equation}
    \mathcal{D} = \{ U_{noise, i}^n | i = 0,1,2,..., \mathbb{N}_x-1, n = 0,1,...,\mathbb{N}_t -1 \} \in \mathbb{R}^{ \mathbb{N}_x \times \mathbb{N}_t},
    \label{e: given data}
\end{equation}
where $\mbNx$ and $\mbNt \in \mathbb{N}$ are the discretization sizes in the spatial and temporal dimensions.  
We assume that the underlying equation has the form of \eqref{unknown eqn}:
\begin{align}
	\frac{\partial u }{\partial t}(x,t) 
 = \sum_{l=1}^L c_l g_l 
 =  \sum_{l=1}^L c_l \frac{\partial^{\alpha_l} f_l(u)}{\partial x^{\alpha_l}}, \ \text{with} \ f_l(u) = u^{\beta_l},
	\label{eq.pde}
\end{align}
where $L$ is the total number of features in the dictionary which consists of feature terms as  $\left\{\frac{\partial^{\alpha_l} f_l(u)}{\partial x^{\alpha_l}}\right\}_{l=1}^L$.  
We use the index $l$ and $(\alpha_l, \beta_l)$ to indicate that the $l$-th feature $g_l$ which is $\frac{\partial^{\alpha_l} u^{\beta_l}}{\partial x^{\alpha_l}}$, the $\alpha_l$-order derivative of the monomial $u^{\beta_l}$  where $\alpha_l,\beta_l$ are nonnegative integers. 

Since we consider PDEs in the frequency domain, we assume that $u$ is a periodic function in $x$ and $t$.  For non-periodic functions, we extend the function to a periodic function, which is to be discussed in Subsection \ref{ss: boundary extension}. 
We define the Fourier transform of $u(x,t)$ as: 
\begin{equation*}
    \cF(u)[\xi_x, \xi_t] = \int_{0}^{T} \int_{0}^{X} u(x, t )e^{ - (\frac{\xi_x {x}}{X}+ \frac{\xi_t {t}}{T}) 2\pi \sqrt{-1}}dx dt.
\end{equation*}
The Fourier transform of  \eqref{eq.pde} with respect to $x$ and $t$  becomes 
\begin{align}
& \frac{2\pi}{T}\xi_t\sqrt{-1}\cF(u) [\xi_x, \xi_t] \nonumber\\
= &  \sum_{l=1}^L c_l \left(\frac{2 \pi }{X}\xi_x\sqrt{-1}\right)^{\alpha_l} \cF(f_l) [\xi_x, \xi_t]  \nonumber
\\
= & \begin{bmatrix}	(\frac{2 \pi }{X}\xi_x\sqrt{-1})^{\alpha_1}\cF(f_1(u)) & \cdots &(\frac{2 \pi }{X}\xi_x\sqrt{-1})^{\alpha_l}\cF (f_2(u)) & \cdots & (\frac{2 \pi }{X}\xi_x\sqrt{-1})^{\alpha_L}\cF(f_L(u))
\end{bmatrix} \cb .
\label{eq.pde.fft}
\end{align}
The first term in \eqref{eq.pde.fft} is the Fourier feature of  $u_t$, and $(\frac{2\pi}{X}\xi_x\sqrt{-1})^{\alpha_l}\cF(f_l(u))$ in the right of \eqref{eq.pde.fft} is that of the $l$-th feature in the library, i.e., $\frac{\partial^{\alpha_l} u^{\beta_l}}{\partial x^{\alpha_l}}$. 

Our idea is to find the support and values of the coefficient vector $\cb$  at the frequency modes $(\xi_x,\xi_t)$ for $\xi_x = 0,1,...,\mbNx-1, \xi_t = 0,1,2...,\mbNt-1$.  We denote $h=\mathcal{H}(\xi_x,\xi_t)$  as the index of the frequency mode  $(\xi_x, \xi_t)$, where $\mathcal{H}$ is a map from the frequency mode $(\xi_x, \xi_t)$ to the unique index $h\in\{1,2,\ldots,H\}$
with
 $H = \#\{(\xi_x,\xi_t): \xi_x = 0,1,...,\mbNx-1, \xi_t = 0,1,...,\mbNt-1\} = \mbNx \mbNt$. 
  The notation $\#$ denotes the cardinality of a set. 
FourierIdent aims to solve 
\textbf{a discrete Fourier system} of \eqref{eq.pde.fft} 
\begin{equation}
    \bF \bc = \bb,
    \label{e: fc = b}
\end{equation}
where 
\[
\bF = (a_{h, l}) \in \mathbb{C}^{H \times L}
, \;\; \bc = (c_l) \in \mathbb{R}^L, \;\; \bb = (b_{h}) \in \mathbb{C}^{H},\]
for 
\begin{equation*}
    a_{h, l} = \left( \frac{2\pi \xi_x}{X} \sqrt{-1} \right)^{\alpha_l}\sum_{p=0}^{\mbNx-1} \sum_{q=0}^{\mbNt-1} \Biggl\{e^{\displaystyle -\left(\frac{2\pi p}{\mathbb{N}_x }\xi_x + \frac{2 \pi q}{\mathbb{N}_t }\xi_t \right)\sqrt{-1}}
\left(U_{noise,p}^q\right)^{\beta_l}  \Biggr\}  \Delta x \Delta t
\end{equation*}
and 
\begin{equation*}
    b_{h, l} = \left( \frac{2\pi \xi_t}{T} \sqrt{-1} \right)\sum_{p=0}^{\mbNx-1} \sum_{q=0}^{\mbNt-1} \Biggl\{e^{\displaystyle -\left(\frac{2\pi p}{\mathbb{N}_x }\xi_x + \frac{2 \pi q}{\mathbb{N}_t }\xi_t \right)\sqrt{-1}}
U_{noise,p}^q \Biggr\} \Delta x \Delta t.
\end{equation*}
Here {$\Delta x = \frac{X}{\mathbb{N}_x },  \Delta t =  \frac{T}{\mathbb{N}_t }$}.   
The $l$-th column of $\bF$ is the discrete Fourier feature associated with the $l$-th feature in the dictionary, and the vector $\boldsymbol{b} $ contains the discrete Fourier feature  for $u_t$.  
The objective of this paper is to find the support and values of this coefficient vector in \eqref{e: fc = b}
$$\boldsymbol{c}=[c_1, c_2,\cdots, c_L]^{\top}$$ from the given data $\mathcal{D}$.

\subsection{Error analysis for Fourier features}
\label{ss: error analysis of fourier features}

We analyze the error of the linear system \eqref{e: fc = b} with discrete Fourier features.  Assume that $\{\epsilon_i^n\}$  in \eqref{e:u+eps} are  \textit{i.i.d} and have zero-mean such that $\mathbb{E}[\epsilon_i^n] = 0$.
Denote the true coefficient vector for the underlying PDE by $\bc$, and the support of the true coefficient vector  by ${\rm Supp}^* = \{ l: \ c_l\neq 0\}$. The true coefficients satisfy the following equation :
\begin{equation}
\frac{2\pi}{T}\xi_t\sqrt{-1}\cF(u) [\xi_x, \xi_t]
=\sum_{l \in \rm{Supp}^* } c_l\left(\frac{2\pi}{X}\xi_x\sqrt{-1}\right)^{\alpha_l} \cF(f_l) [\xi_x, \xi_t],
\label{e: pde fft true}
\end{equation}
for $\xi_x = 0,1,...,\mbNx-1, \xi_t = 0,1,...,\mbNt-1$.

The $L_\infty$-residual of  \eqref{e: fc = b}  consists of two errors: one is the discretization error of Fourier features and the other is the error from noise:
\begin{equation*}
    e = \|\bF \bc - \bb\|_{\infty} \leq {e}_{\rm{Fourier}} + {e}_{\rm{noise}}.
\end{equation*}
The discretization error of Fourier features is 
\begin{align*}
{e}_{\rm{Fourier}}  = &  
\max_{\xi_x,\xi_t}\Bigg|\sum_{l \in \rm{Supp}^*}c_l
\left( \frac{2\pi \xi_x}{X} \sqrt{-1} \right)^{\alpha_l}
\sum_{p=0}^{\mbNx-1} \sum_{q=0}^{\mbNt-1} \Biggl\{e^{-\left(\frac{2\pi p}{\mathbb{N}_x }\xi_x + \frac{2 \pi q}{\mathbb{N}_t }\xi_t \right)\sqrt{-1}}
\left(U_p^q\right)^{\beta_l} \Biggr\}  \nonumber\\
& -
\left( \frac{2\pi \xi_t}{T} \sqrt{-1} \right)\sum_{p=0}^{\mbNx-1} \sum_{q=0}^{\mbNt-1} \Biggl\{e^{-\left(\frac{2\pi p}{\mathbb{N}_x }\xi_x + \frac{2 \pi q}{\mathbb{N}_t }\xi_t \right)\sqrt{-1}}
U_p^q \Biggr\} \Bigg| \frac{XT}{\mathbb{N}_x \mathbb{N}_t},   
\end{align*}
and decreases to $0$ as the data sampling grid is refined:
$$\lim_{\Delta x,\Delta t \rightarrow 0}{e}_{\rm{Fourier}} = 
 \max_{\xi_x,\xi_t} \left| \sum_{l\in {\rm Supp}^*} c_l \left(\frac{2 \pi }{X}\xi_x\sqrt{-1}\right)^{\alpha_l} \cF(u^{\beta_l}) [\xi_x, \xi_t] - \frac{2\pi}{T}\xi_t\sqrt{-1}\cF(u) [\xi_x, \xi_t] \right| = 0. $$
The error from the noise is 
\begin{align}
{e}_{\rm{noise}} = & \max_{\xi_x, \xi_t} |e_{\rm{noise}}[\xi_x,\xi_t]|\nonumber\\
    = & \max_{\xi_x, \xi_t}\Bigg|\sum_{l \in \rm{Supp}^*}c_l
\left( \frac{2\pi \xi_x}{X} \sqrt{-1} \right)^{\alpha_l}
\sum_{p=0}^{\mbNx-1} \sum_{q=0}^{\mbNt-1} \Biggl\{e^{  -\left(\frac{2\pi p}{\mathbb{N}_x}\xi_x + \frac{2 \pi q}{\mathbb{N}_t}\xi_t \right)\sqrt{-1}}  \left(\left(U_{noise,p}^q\right)^{\beta_l} - \left(U_{p}^q\right)^{\beta_l} \right)\Biggr\} \nonumber\\
  & -
\left( \frac{2\pi \xi_t}{T} \sqrt{-1} \right)\sum_{p=0}^{\mbNx-1} \sum_{q=0}^{\mbNt-1} \Biggl\{e^{  -\left(\frac{2\pi p}{\mathbb{N}_x}\xi_x + \frac{2 \pi q}{\mathbb{N}_t}\xi_t \right)\sqrt{-1}}
\epsilon_p^q \Biggr\} \Bigg| \frac{XT}{\mathbb{N}_x \mathbb{N}_t}, 
\label{e: e_noise}
\end{align}
where ${e}_{\rm{noise}}[\xi_x,\xi_t]$ denotes the residual at the frequency mode $(\xi_x,\xi_t)$.
In the following, we prove an upper bound for the error ${e}_{\rm noise}$ resulted from noise. 
\begin{theorem}
\label{theoremerror}
Consider the differential equation  in \eqref{eq.pde} whose Fourier form is  \eqref{e: pde fft true}.
Assume that the given data $\mathcal{D}$ in \eqref{e: given data} are contaminated by i.i.d.  noise:  $\{\epsilon_i^n\}$ are i.i.d  bounded random variables with $\mathbb{E}[\epsilon_i^n] = 0$ 
and  $|\epsilon_i^n| \le  \epsilon$ for some $\epsilon >0$. The area of $\Omega$ is denoted by $|\Omega| = XT$. 
Then the error ${e}_{\rm noise}$   satisfies
\begin{equation}
{e}_{\rm{noise}} \leq   \frac{XT}{\mathbb{N}_x\mathbb{N}_t}S \epsilon + 
\mathcal{O}(\epsilon^2),
\label{eq:thm1error}
\end{equation}
where 
\begin{equation}
S =  \max_{\xi_x, \xi_t}
 \Bigg|  \sum_{p=0}^{\mbNx-1}  \sum_{q=0}^{\mbNt-1}   
e^{-\left(\frac{2\pi p}{\mathbb{N}_x }\xi_x + \frac{2 \pi q}{\mathbb{N}_t }\xi_t \right)\sqrt{-1}}  \left( \sum_{l \in
\rm{Supp}^*}c_l
\left( \frac{2\pi \xi_x}{X} \sqrt{-1} \right)^{\alpha_l}\beta_l(U_p^q)^{\beta_l-1} 
- \frac{2\pi \xi_t}{T}
\sqrt{-1} \right)
\Bigg| .
\label{eq:S}
\end{equation}
\end{theorem}
\begin{proof}
Following the definition of ${e}_{\rm{noise}}$ in \eqref{e: e_noise}, we have
\begin{align*}
{e}_{\rm{noise}} 
= &  
    \max_{\xi_x,\xi_t}\Bigg|\sum_{l \in \rm{Supp}^*}c_l
\left( \frac{2\pi \xi_x}{X} \sqrt{-1} \right)^{\alpha_l}
\sum_{p=0}^{\mbNx-1} \sum_{q=0}^{\mbNt-1} \Biggl\{e^{-\left(\frac{2\pi p}{\mathbb{N}_x }\xi_x + \frac{2 \pi q}{\mathbb{N}_t }\xi_t \right)\sqrt{-1}} \left(\beta_l \left(U_p^q\right)^{\beta_l-1} \epsilon_p^q+\mathcal{O}((\epsilon_p^q)^2)\right)\Biggr\} \nonumber\\
 & -
\left( \frac{2\pi \xi_t}{T} \sqrt{-1} \right)\sum_{p=0}^{\mbNx-1} \sum_{q=0}^{\mbNt-1} \Biggl\{e^{-\left(\frac{2\pi p}{\mathbb{N}_x }\xi_x + \frac{2 \pi q}{\mathbb{N}_t }\xi_t \right)\sqrt{-1}}
\epsilon_p^q \Biggr\} \Bigg| \frac{XT}{\mathbb{N}_x\mathbb{N}_t} 
\leq   \frac{XT}{\mathbb{N}_x\mathbb{N}_t}S \epsilon   + 
 \mathcal{O}(\epsilon^2),
\end{align*}
where  $S$ is defined in \eqref{eq:S}.
\end{proof}

First, Theorem \ref{theoremerror} shows that the error $\be_{\rm{noise}}$ scales linearly with respect to the noise level $\epsilon$.
In comparison,  numerical differentiation of the features in the physical domain can give error as 
\begin{equation*}
\mathcal{O}\left(\Delta t + \Delta x^{p+1-r} + \frac{\epsilon}{\Delta t} +\frac{\epsilon}{(\Delta x) ^r} \right)
\end{equation*}
where $r$ is the highest order of derivatives for the active features in the true support, and  numerical differentiation is carried by interpolating the data by a $p$th order polynomial \cite{kang2021ident, he2022robust}. The noise is magnified by $1/\Delta t$ and $1/(\Delta x)^r$, which shows the challenges of dealing with noisy data. 
Theorem \ref{theoremerror} shows the formulation of Fourier features \eqref{eq:thm1error} is more robust against noise in comparison with  numerical formulation of differential features. 
In the weak formulation, such as WeakIdent\cite{tang2023weakident}, the error (resulted from noise) of the linear system also scales linearly with respect to  noise, 
\begin{align}
{e}_{\rm noise}^{\rm weak} & \le 
     \bar{S}^*   |\Omega_h|   \epsilon + \mathcal{O}\left(\epsilon^2\right), \label{eq:weakidenterror}
     \\
     & \text{with}\   \bar{S}^*  =\max_h \sup_{(x_j,t^k) \in \Omega_h}
\bigg|
\sum_{l\in \rm{Supp}^*}(-1)^{\alpha_l}{c_l}
{\beta_l}(U_j^k)^{ \beta_l - 1}
\frac{\partial^{\alpha_l} \phi }{\partial x^{\alpha_l}}(x_j, t^k)
-  \frac{\displaystyle \partial \phi}{\displaystyle \partial t}  (x_j, t^k) \nonumber
\bigg|,
\end{align}
where $h$ is an index of the test function, $\Omega_h$ is the support of the $h$-th test function, and $|\Omega_h|$ is the area of $\Omega_h$ \cite[Theorem 1]{tang2023weakident}.  The error bounds in both  \eqref{eq:thm1error} and \eqref{eq:weakidenterror} scale linearly with respect to the noise level $\epsilon$, which demonstrates that the numerical formulations of  weak features and Fourier features are both robust to noise. 

Secondly, the major difference between \eqref{eq:thm1error} and \eqref{eq:weakidenterror} is that, the error in WeakIdent is local depending on the local support $\Omega_h$ of the $h$-th test function, while the error in FourierIdent is global.

\section{Fourier feature denoising and core regions of features} \label{s: denoising a feature}

One of the main difficulties of FourierIdent is that frequency responses may 
{have large magnitude}, which is different from the linear system generated in the physical domain.  As in \cite{kang2021ident, he2022robust, tang2023weakident}, even in physical domain denoising is very important for coefficient identification.   
In this section, we propose a  denoising process for FourierIdent using (i) Fourier feature denoising as in the physical domain, and defining partitions  of frequency domain via (ii) the meaningful data region, and (iii) the core regions of features in the frequency domain.  We give details of denoising, then present the main algorithm of  FourierIdent  in Section \ref{s: feature selection and coefficient recovery}.

\subsection{Denoising Fourier features}
We first denoise  Fourier features by applying convolution with a Gaussian-shape  kernel $\phi(x,t)$. 
By the convolution theorem, we conduct all computations of  convolution in the frequency domain.  This can be also applied in a restricted domain. For example, let $\Lambda$ be a smaller region in the frequency domain, and the matrix $\bF_{\Lambda}$ and the vector $ \bb_{\Lambda}$ be restricted to the entries from the region $\Lambda$.
We let  $(\xi_x,\xi_t) \in \Lambda$, and $h$ be the associated index such that  
{$h=\mathcal{H}(\xi_x,\xi_t)$, which is also the row index of the frequency mode $(\xi_x, \xi_t)$ in $\mathcal{S}(\bF)$.}
We define the smoothing operator $\mathcal{S}$ on the matrix $\bF_{\Lambda}$ as follows:
\begin{align}\label{e:denoisingS}
& [\cS({\bF})_{\Lambda}]_{h,l} \\
= &\left( \frac{2\pi \xi_x \sqrt{-1}}{X}  \right)^{\alpha_l}\left(\sum_{p=0}^{\mbNx-1} \sum_{q=0}^{\mbNt-1} e^{-\left(\frac{2\pi p}{\mathbb{N}_x}\xi_x + \frac{2 \pi q}{\mathbb{N}_t}\xi_t \right)\sqrt{-1}}
\left(U_{noise,p}^q\right)^{\beta_l} \right)  \cdot \left(\sum_{p=0}^{\mathbb{N}_x-1} \sum_{q=0}^{\mathbb{N}_t-1} e^{-\left(\frac{2\pi p}{\mathbb{N}_x}\xi_x + \frac{2 \pi q}{\mathbb{N}_t}\xi_t \right)\sqrt{-1}}\Phi_p^q\right) , \nonumber
\end{align}
where $(\xi_x,\xi_t) \in \Lambda$, and $h$ and $l$ are row and column index of the smoothed matrix $\cS(\bF)_{\Lambda}$,
\[
\phi(x,t) = \left(1 - \left(\frac{x}{m_x\Delta x}\right)^2\right)^{p_x}\left(1 - \left(\frac{t}{m_t\Delta_t}\right)^2\right)^{p_t} 
\text{for } (x,t) \in (-m_x\Delta x, m_x\Delta x) \times (-m_t\Delta t, m_t\Delta t)
\]
and $\Phi_p^q\;\;(p = 0,...,\mathbb{N}_x-1, q = 0,...,\mathbb{N}_t-1)$ gives a  discrete evaluation of $\phi(x,t)$ on the grid point $(x_p, t^q)$. 
The  operation with $\phi(x,t)$ in \eqref{e:denoisingS} smoothes the data, because the point-wise multiplication between $\Phi$ and $U$ in the frequency domain is equivalent to the convolution between $\phi$ and $u$ in the physical domain. 
This denoising process is motivated by the weak form of feature as in \cite{messenger2021weak, tang2023weakident}, and we apply the same method to determine the nonnegative integers $m_x, m_t, p_x, p_t$ related with the decay of the test function $\phi$, from the given data. The idea is to match the shape of the smoothing kernel $\phi$ to the shape of a Gaussian function, such that the noise region will be mostly in the tail of the Gaussian.  
Another aspect is the smoothness of $\phi$ that $m_x, m_t, p_x, p_t$ are chosen to guarantee all the features in the dictionary,   including the highest derivative terms, are at least continuous.

\subsection{The meaningful data region $\Lambda$ in the frequency domain}
\label{ss:meaningfulRegion}

We separate the frequency domain into the meaningful data region and the core regions of features for further noise separation.  We first utilize a theoretical bound for the decay of Fourier coefficients to define a meaningful data region $\Lambda$. 

\begin{lemma}\label{L:decayrate}
Let
$u \in L^2([0,1])$ be continuous
and denote its Fourier transform by $\cF(u)$ with  
$\cF(u)[\xi] =  \int_0^{1} u(x) e^{-2\pi \sqrt{-1}\xi x}dx$.
If $u^{(p-1)} = \frac{\partial^{p-1}u}{\partial x^{p-1}}$ is continuous and  $u^{(p-1)} \in L^2([0,1])$,
then the Fourier coefficients $\cF(u^{(q)})$ of  the $q$-th  ($q \leq p$) order derivative $\frac{\partial^q u}{\partial x^q}$ satisfies 
\begin{equation}
\cF(u^{(q)})[\xi] \propto \frac{1}{\xi^{p+1 - q}}, \label{e:decayrate}
\end{equation}
where $\propto$ denotes the proportional relation. 
\begin{proof}
It is shown in \cite{trefethen2000spectral}  
that  the Fourier coefficients of $u$ satisfies $\cF(u)[\xi] \propto \frac{1}{\xi^{p+1}}$. 
    Since $\cF(u^{(q)})[\xi] = (2\pi\sqrt{-1} \xi)^q \cF(u)[\xi]$, we obtain that $\cF(u^{(q)})[\xi]  \propto \frac{1}{\xi^{p+1 - q}}$.
\end{proof}
\end{lemma}

Following Lemma \ref{L:decayrate}, we  assume that PDE solution satisfies  
$
\cF(u)[\xi_x,\xi_t] \propto \frac{1}{\xi_x^{c_1}}$, and $ \cF(u)[\xi_x,\xi_t] \propto \frac{1}{\xi_t^{c_2}}
$
for some  $c_1, c_2 >0$.
Using this decay rate, we  partition the frequency domain into two regions. One region follows a proper decay rate which gives the meaningful data region $\Lambda$ for differential equation identification, and the other high frequency region indicates the noise from the given data $\mathcal{D}$.

For simplicity, we use $U \in \mathbb{C}^{N_x \times N_t}$ to denote the given data. We assume the data are periodic in both time and space by the augmentation along the temporal domain in Section \ref{ss: boundary extension}.
After augmentation, the spatial and temporal dimensions are updated to $N_x = \mathbb{N}_x$ and $N_t = 2 \mathbb{N}_t - 2$. 
   
We take the two-dimensional Fourier Transform of the noisy data and  accumulate the  responses in the $\xi_x$ and $\xi_t$ dimension (frequency index $\xi_x$ and $\xi_t$) respectively, and exploit the same decay rate in both dimensions:
\begin{equation*}
    \sum_{\xi_t = 0}^{N_t-1}\big|\mathcal{F}(U)[\xi_x,\xi_t]\big| \propto \frac{1}{\xi_x^{c_1}}, \quad 
      \sum_{\xi_x = 0}^{N_x-1}\big|\mathcal{F}(U)[\xi_x,\xi_t]\big| \propto \frac{1}{\xi_t^{c_2}}.
\end{equation*}
This can be written in a linear form using the log function:
\begin{equation}
    \log\left(\sum_{\xi_t = 0}^{N_t-1}\big|\mathcal{F}(U)[\xi_x,\xi_t]\big|\right) \propto \;\;  -c_1\log(\xi_x) + d_1.
    \label{e:decay_logform}
\end{equation}

We find a transition frequency mode  ${a_x}$ separating two regions (meaningful vs noise), by fitting a linear relation for low frequency modes ($\xi_x\leq a_x$), and fitting a flat line for high frequency modes ($\xi_x>a_x$) in the $\log$-$\log$ scale: 
\begin{align}
    &a_x^* = \argmin_{{a_x}\in \{0,..., \lfloor{N_x/2} \rfloor\}} \Gamma(a_x), \nonumber\\
    &\Gamma(a_x)=
    \left(
    \min_{a, b \in \mathbb{R}} \sum_{\xi_x \in \{ 0,1,...,{a_x}\}}\left(a\log(\xi_x)+b -{ \log[y(\xi_x)]}\right)^2
    \right) + 
    \sum_{\xi_x = {a_x}+1}^{ \lfloor{N_x/2} \rfloor}\left(\log[\bar{y}({a_x}+1)] - \log[y(\xi_x)]\right)^2
    ,
    \label{e: cost fun of decay fit}
\end{align}
with
\begin{align*}
&y(\xi_x) = \sum_{\xi_t = 0}^{N_t-1}\big|\mathcal{F}(U)[\xi_x,\xi_t]\big|,\quad \bar{y}({a_x}+1)=\frac{1}{\lfloor{N_x/2} \rfloor-{a_x}}\sum_{\xi_x={a_x}+1}^{\lfloor{N_x/2} \rfloor}\sum_{\xi_t = 0}^{N_t-1}\big|\mathcal{F}(U)[\xi_x,\xi_t]\big|.
\end{align*}

The first term in \eqref{e: cost fun of decay fit}  fits the decay rate \eqref{e:decay_logform} and the second term fits the flat noise region. 
This is based on the assumption that low frequency data are 
dominated by the actual dynamics of the differential equation, while the high frequency data are mostly 
dominated by noise. 
Thus, a transition frequency mode can be detected by the change  of the relation between $\log\left(\sum_{\xi_t = 0}^{N_t-1}\big|\mathcal{F}((U))[\xi_x,\xi_t]\big|\right)$ and $\log(\xi_x)$. 
The transition frequency mode in time, denoted by ${a_t^*}$, can be obtained similarly.

\begin{figure}[t!]
    \centering
    \begin{tabular}{cc}
    (a) & (b) \\
      \includegraphics[width = 0.5\textwidth]{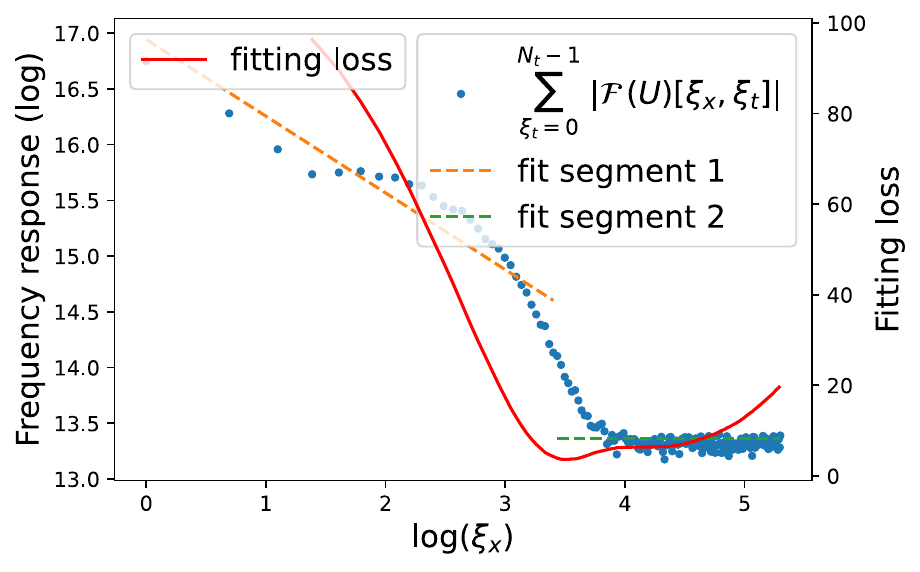}  &
      \includegraphics[width = 0.5\textwidth]{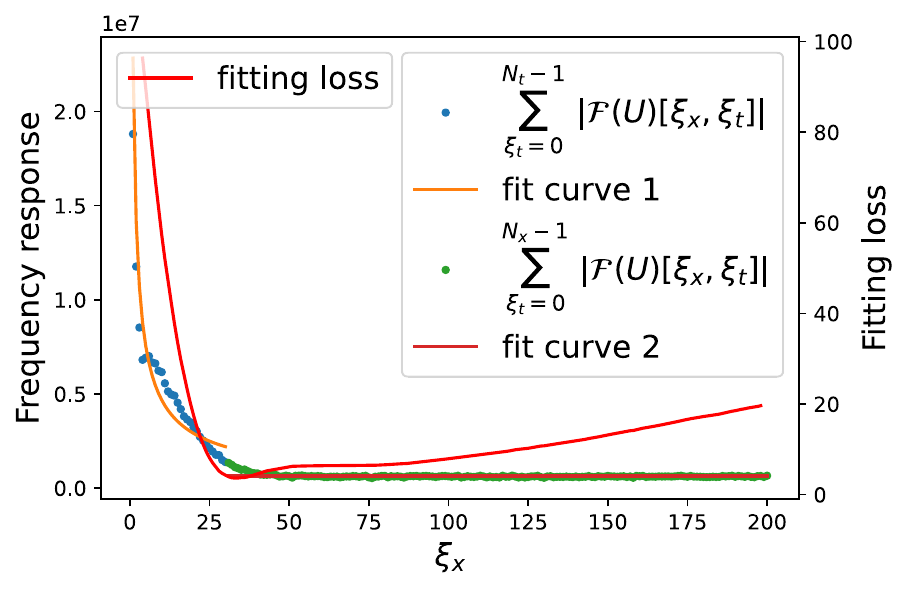} \\
    \end{tabular}
    \caption{Decay fit in the frequency and physical domains. (a) Blue dots are the given data, orange and green dotted lines are the fitting lines to find the transition frequency mode $a_x^*$ \eqref{e: cost fun of decay fit} in log-log scale.  
(b) the same plot in physical domain and original unit.  Red curves in both graphs show the fitting loss. 
}
    \label{fig: decay fit}
\end{figure}

In Figure \ref{fig: decay fit} (a) we show the accumulated frequency responses (blue dots), the linear fitting curve at the low-frequency modes (dash lines), and the total loss function $\Gamma$ of \eqref{e: cost fun of decay fit} (red lines). 
The $x$-axis and left-$y$-axis  provide the scales of points  $(x_i,y_i)$ in the linear fit. The right-$y$-axis  provides the scale of the loss function $\Gamma$ in \eqref{e: cost fun of decay fit}.
In (b),  visualization of the fitted curves is shown in the original scale (instead of $\log$ scale) in the frequency domain.   
The transition mode is detected as  $a_x^* = 27$ from \eqref{e: cost fun of decay fit}.

With the transition frequency modes $a_x^* $ and $a_t^*$ in  $\xi_x$ and $\xi_t$ directions respectively, we partition the frequency domain into  two regions: the \textbf{meaningful data region} $\Lambda$ and noise regions, where  
\begin{equation*}  
  \Lambda = \{ (\xi_x,\xi_t): \xi_x = 0,1,...,a_x^*-1, \xi_t =  0,1,...,a_t^*-1 \}  
\end{equation*}
and the complement $\Lambda^{\complement}$ is the noise region.  This is motivated from the theoretical decay rate of the Fourier coefficient of $U$.

When the given data are noisy, each Fourier feature $\cF\{g_l\}$ is also affected by noise. The region with a high response of $\cF\{g_l\}$ usually contains more signal information than the region with a low response of $\cF\{g_l\}$.
We further define the \textbf{core region of feature} for $g_l$, denoted as $\mathcal{V}(g_l)$ such that: 
    \begin{align}
        &\mathcal{V}(g_l)=\left(\mathcal{V}_{\mathcal{R}}(g_l), \mathcal{V}_{\mathcal{I}}(g_l) \right) \label{e: v(gl)} \;\; \mbox{ with}\\
        &\mathcal{V}_{\mathcal{R}}(g_l)=\left\{h=\mathcal{H}(\xi_x,\xi_t):(\xi_x,\xi_t) \in \Lambda, \  |\mathcal{R}\left(\big[\cS (\bF)\big]_{ h, l
        }\right)|\geq \beta_{g_l} \right\}, \ \  \nonumber \\
        &\mathcal{V}_{\mathcal{I}}(g_l)=\left\{h = \mathcal{H}(\xi_x,\xi_t): (\xi_x,\xi_t)\in{\Lambda},\  |\mathcal{I}\left(\big[\cS (\bF)\big]_{ h, l
        }\right)|\geq \beta_{g_l} \right\}, \nonumber
    \end{align}
    where $\mathcal{R}(\cdot)$ and $\mathcal{I}(\cdot)$ denote the real part and imaginary part of its argument, respectively.
Here $\mathcal{V}(g_l)$ contains two ordered sets, in which the first set denotes the high response region for the real part of the Fourier feature for $g_l$, and the second set denotes the high response region for the imaginary part.

For the feature $g_l$,   
we pick the high responses among the frequencies in $\Lambda$ to be in the core region of feature $g_l$, and consider low responses as noise and remove them from the system \eqref{e: fc = b}.   We refer to Subsection \ref{s: choice of beta} for more details on how to choose the threshold $\beta_{g_l}$.
For the feature  $u_t$, 
we consider the smoothed frequency response restricted to $\Lambda$,  $\cS({\bb})_{\Lambda}$, and  its magnitude  at each frequency.  We  take the high response region as the core region determined by a threshold $\beta_{u_t}$. 
{Specifically, the core region of $u_t$, $\mathcal{V}(u_t)$, is defined similarly as $\mathcal{V}(g_l)$ in \eqref{e: v(gl)}, where $\beta_{g_l}$ is replaced by $\beta_{u_t}$ and $\mathcal{S}(\bF)_{h,l}$ is replaced by $\mathcal{S}(\bb)_{h}$}.

Using the core region of feature $u_t$, we construct the linear system  as

\begin{equation}    \overline{\cS({\bF})_{\mathcal{V}(u_t)}} \bc  = \overline{\cS({\bb})_{\mathcal{V}(u_t)}},
    \label{e: Fc = b on A u_t with phi}
\end{equation}
where 
\[
\overline{\cS({\bF})_{\mathcal{V}({u_t})}} = 
\begin{bmatrix}
    \mathcal{R}\left(\cS (\bF)_{  \mathcal{V}_{\mathcal{R}}(u_t)}\right)\\
    \mathcal{I}\left(\cS (\bF)_{  \mathcal{V}_{\mathcal{I}}(u_t)}\right)
\end{bmatrix}
\;\;
\text{and}\;\;
\overline{\cS({\bb})_{\mathcal{V}({u_t})}} = 
\begin{bmatrix}
    \mathcal{R}\left(\cS (\bb)_{  \mathcal{V}_{\mathcal{R}}(u_t)}\right)\\
    \mathcal{I}\left(\cS (\bb)_{  \mathcal{V}_{\mathcal{I}}(u_t)}\right)
\end{bmatrix}
.
\]
Here we use $\cS (\bF)_{  \mathcal{V}_{\mathcal{R}}(u_t)}$  to denote the submatrix of $\cS (\bF)$ with rows restricted to the ones indexed by ${  \mathcal{V}_{\mathcal{R}}(u_t)}$.
The overline denotes the vertical stacking of real and imaginary responses. This leads to a reduced linear system \eqref{e: Fc = b on A u_t with phi} that is defined only on the core region of the feature $u_t$.  We represent the coefficients identified within the core region of $u_t$ as 
\begin{equation}\bc_{\mathcal{V}(u_t)} = \rm{LeastSquare}(\widetilde{\cS(\bF)_{\mathcal{V}(u_t)}}, \widetilde{\cS(\bb)_{\mathcal{V}(u_t)}}).
\label{e: coefficient on u_t}
\end{equation} 
Here we use the error-normalized Fourier feature matrix  $\widetilde{\cS(\bF)_{\mathcal{V}(u_t)}}$ and the error-normalized dynamic variable $\widetilde{\cS(\bb)_{\mathcal{V}(u_t)}}$ 
defined in \eqref{e: least square fit eqn}  to  solve the least square problem in \eqref{e: coefficient on u_t} (details presented in Subsection \ref{ss: error normalization}). 

\section{Fourier features for Identifying differential equations (FourierIdent)}
\label{s: feature selection and coefficient recovery}
\begin{figure}[t!]
\centering
\begin{tikzpicture}[node distance=2cm]
\node (input) [smallblockpurple]{Construct the Fourier system \eqref{e: fc = b}, from the given data $\mathcal{D}$ \eqref{e: given data}.};
\node (featureprocess) [medianblockpurple][right of=input, xshift = 1.5cm]{\textbf{[Step 1]} \\ Fourier feature denoising and construct $\overline{\cS(\boldsymbol{F})_{\mathcal{V}(u_t)}}$ and  $\overline{\cS(\boldsymbol{b})_{\mathcal{V}(u_t)}}$ in \eqref{e: Fc = b on A u_t with phi}.};
\node (sp) [medianblockpurple][right of=featureprocess, xshift =2cm]{\textbf{[Step 2]} \\ For 
$k = 1,2,...,K$, apply SP and group trimming to obtain the support 
$\mathcal{A}_k$.}; 
\node (grouptrim) [medianblockpurple, right of=sp, xshift = 2cm]{\textbf{[Step 3]} \\ 
Chose the best sparsity $k^*$ as in \eqref{e:energyCoreRegion}, and use 
$\mathcal{V}^*$ in \eqref{e: coeff selection} to compute the coefficient $\bc_{\mathcal{V}^*}$. };
\node (picksupport) [smallblock_final, right of=grouptrim,xshift = 1.2cm] {Output $\bc_{\mathcal{V}^*}$};
\draw [arrow] (input.east) --  (featureprocess);
\draw [arrow] (featureprocess) --  (sp);
\draw [arrow] (sp) --  (grouptrim);
\draw [arrow] (grouptrim) --  (picksupport);
\end{tikzpicture}
\caption{The flowchart of FourierIdent. 
A discrete Fourier system \eqref{e: fc = b} is constructed from the given data.   
In [Step 1], Fourier features are denoised and a smaller system \eqref{e: Fc = b on A u_t with phi} is constructed. %
In [Step 2], we apply Subspace Pursuit to obtain initial supports for each sparsity level $k = 1,2,..., K$ and new group trimming is applied.  
In [Step 3], the best coefficient is computed from the energy \eqref{e:energyCoreRegion} based on core region. } 
\label{F:flowchart}
\end{figure}
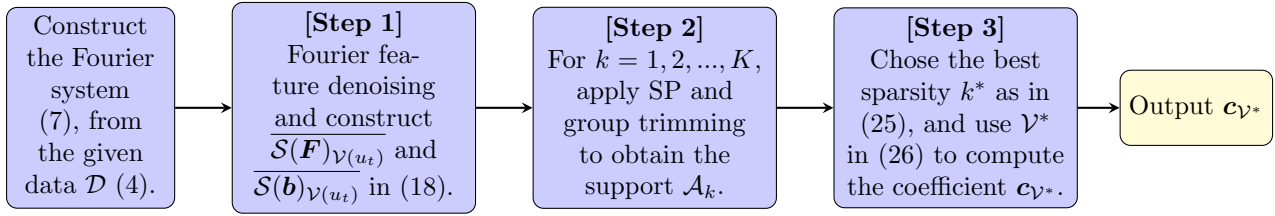

From the given data, Fourier features are denoised and a smaller  linear system with reduced size is constructed in Section \ref{s: denoising a feature}.  We utilize Subspace Pursuit \cite{dai2009subspace} for the coefficient support identification, and we propose (i) a new group trimming for stable support recovery, and (ii) a new energy based on the core regions of features for the  coefficient identification.
There are three steps to the proposed FourierIdent, and Figure \ref{F:flowchart} shows the flowchart to illustrate the process.
\begin{enumerate}
\item From the given input data, we construct a  denoised linear system with  $\overline{\cS(\boldsymbol{F})_{\mathcal{V}(u_t)}}$ and  $\overline{\cS(\boldsymbol{b})_{\mathcal{V}(u_t)}}$ in \eqref{e: Fc = b on A u_t with phi}.  This is generated by the denoising operation $\cS$  in \eqref{e:denoisingS}, the discrete Fourier system \eqref{e: fc = b} and using the core region  of $u_t$, $\mathcal{V}(u_t)$, to reduce the size and denoise the system. 

\item  For each sparsity level $k =1, 2, \dots, K$, we apply Subspace Pursuit \cite{dai2009subspace} on a column normalized system of $\overline{\cS({\bF})_{\mathcal{V}(u_t)}}  ^{\dagger}$, and  $\overline{\cS({\bb})_{\mathcal{V}(u_t)}} ^{\dagger}$ and get an initial support $\mathcal{A}_k^0$.  (Here $^{\dagger}$ denotes the column normalization, which each column is normalized independently.) 
We apply a new group trimming to remove features which minimally contribute to the dynamics of the system. This is applied iteratively, until the support converges to $\mathcal{A}_k$ 
for  each $k$. 

\item To choose the optimal coefficient, we propose a new energy based on the core regions of features.  
First, the optimal sparsity is determined by choosing $k^*$ which gives the minimum residual on the union of the core regions of features identified in $\mathcal{A}_k$.  Secondly, we identify the optimal coefficients, by finding the coefficients which gives the  minimum residual among least-square fits on the core region of each feature in $\mathcal{A}_k$.
\end{enumerate}

\begin{algorithm}[t!]
\caption{FourierIdent Algorithm} \label{alg: FourierIdent}
\begin{algorithmic}
\STATE \textbf{Input:} Given data \eqref{e: given data}.  
\STATE{[Step 1] Construct the discrete linear system $\overline{\cS({\bF})}_{\mathcal{V}(u_t)}, \overline{\cS({\bb})}_{\mathcal{V}(u_t)}$ as in \eqref{e: Fc = b on A u_t with phi}}.
\STATE{[Step 2]}\\
\FOR{k = 1,2,...,K}
\STATE{Run SP$(\overline{\cS({\bF}
)_{\mathcal{V}(u_t)}} ^{\dagger}, \overline{\cS({\bb}
)_{\mathcal{V}(u_t)}} ^{\dagger} ,k)$\, using SP \cite{dai2009subspace} and set $j=0$ \;} 
\STATE{
Compute the coefficients $\bc_{\mathcal{V}(u_t)}$ by the Least Square in \eqref{e: coefficient on u_t}\;}
\WHILE{$\mathcal{A}_k^j$ not convergent } 
\STATE{
Perform group trimming  in \eqref{e: group trim criteria} with  $\mathcal{T}=0.08$ 
and set $j=j+1$ \;} 
\STATE{
Update the contribution score $s_l$ in \eqref{e:contribution score}, using the updated  coefficients  $\bc_{\mathcal{V}(u_t)}$ in \eqref{e: coefficient on u_t}\;}
\ENDWHILE
\ENDFOR
\STATE{[Step 3 - 1] Determine the optimal sparsity $k^*$ according to \eqref{e:energyCoreRegion} and the support $\mathcal{A}_{k^*}$. 
}
\STATE{[Step 3 - 2] Compute $\bc^*$ 
by minimizing 
\eqref{e: c optimal}. 
}
\STATE{\textbf{Output:} $\boldsymbol{c}^* \in \mathbb{R}^L$ such that $\bF \boldsymbol{c}^*  \approx \bb$ .}
\end{algorithmic}
\end{algorithm}

\subsection{Subspace Pursuit (SP) and Group trimming in [Step 2]} 
\label{ss: group trimming}

With the denoised Fourier features, we apply Subspace Pursuit (SP) \cite{dai2009subspace} to obtain an initial support $A_k^0$ with each sparsity level $k$. The goal is to find a $k$-sparse solution of the linear system $\Phi \bx = \bb$:
\begin{align*}
    \min_{\bx: \|\bx\|_0=k} \|\Phi\bx-\bb\|_2^2. 
\end{align*}
One inputs a column normalized matrix $\Phi$, a vector $\bb$, and a sparsity $k$, and Subspace Pursuit ${\rm SP}(\Phi,\bb,k)$ outputs a $k$-sparse solution in a greedy way (See \cite{dai2009subspace} for details).  

For each sparsity level $k = 1, \dots, K$, we compute ${\rm SP}(\overline{\cS({\bF})_{\mathcal{V}(u_t)}}  ^{\dagger},\overline{\cS({\bb})_{\mathcal{V}(u_t)}} ^{\dagger},k)$, and denote the recovered support as $A^0_k$.
To further remove the insignificant features in $A_k^0$, we trim the features which have minimum contributions to the dynamics, measured by the \textit{contribution score};
\begin{equation} 
    s_l= s(g_l) = |c_l| \cdot \| \overline{\mathcal{S}(\bF)_{\mathcal{V}(u_t),l}}\|_2 \;.
    \label{e:contribution score}
\end{equation}
{Here $  \overline{\mathcal{S}(\bF)_{\mathcal{V}(u_t),l}}$ is  
the $l$-th column of $ \overline{\mathcal{S}(\bF)}$ with rows restricted to the core region $\mathcal{V}(u_t)$. } 
The value $c_l$ is the $l$-th entry of the identified coefficient from 
the least-square fit of  
$\overline{\cS( {\bF})_{\mathcal{V} (u_t)}}$  and $\overline{\cS( {\bb})_{\mathcal{V} (u_t)}}$.

To improve the efficiency of the algorithm,  we propose to {\bf trim features as a group}. We remove a set of the least relevant features for a given sparsity level $k$ using a threshold $\mathcal{T}$.  In this paper, we fix this threshold to be  $\mathcal{T} =0.008$.  The group-trimming is used as follows:
\begin{itemize}
    \item \textbf{Re-order features by contribution scores.}  We arrange the features in ascending order, denoted as $1', 2', ..., k'$, based on their respective contribution scores, represented as $s_{1'}\leq s_{2'}\leq ...\leq s_{k'}$.
    \item \textbf{Remove the low contributing terms as a group.}  We identify the largest value $k_{\max}'$ satisfying condition 
    \begin{equation}
    \frac{\sum_{j = 1'}^{k_{\max}'}s_{j}}{\sum_{k \in  \mathcal{A}_k^0} s_{k}}
     < \mathcal{T}  \leq \frac{\sum_{j = 1'}^{k_{\max}' + 1} s_{j}}{\sum_{k \in  \mathcal{A}_k^0} s_{k}},
    \label{e: group trim criteria}
\end{equation} 
and remove features with contribution scores equal to or lower than $s_{k_{\max}'}$, i.e., $s_l \leq s_{k_{\max}'}$ from the support set $\mathcal{A}_k^i$ when $k_{\max}' > 0$, 
where $\mathcal{A}_k^i$ denotes the support at the $i$th iteration after the $i$th group trimming.  
This step reduces the size of the support and the support becomes $\mathcal{A}_k^{i+1}$. 
Here $\mathcal{T}$ helps to remove a subgroup of features that contributes less than a threshold (which is set to be 8\% in this paper) 
overall  compared to all features in the  support 
$\mathcal{A}_k^i$. 
This $\mathcal{T}$ is a measure of the relative significance of a smaller subgroup of trimmed features compared to the whole set of features.
\item \textbf{Trimming in each iteration. } This group trimming is applied iteratively within each sparsity level $k$, and the converged support set is assigned as $\mathcal{A}_k$.  Note that equation \eqref{e: group trim criteria} changes in each iteration and must be recomputed. 
\end{itemize}

The group trimming is a grouped version of trimming introduced in WeakIdent \cite{tang2023weakident}. 
The group trimming improves efficiency by removing insignificant terms as a group and gives stable results for different types of equations with various levels of noise.

\subsection{ Identification of coefficients from the core regions of features in [Step 3]}
\label{ss: support selection active modes}
Once we have a collection of support sets $\mathcal{A}_k$'s for  $k=1,\dots, K$, we choose the optimal sparsity $k^*$ (and the associated support $\mathcal{A}_{k^*}$) which minimizes an energy based on the core regions of features.  This energy consists of two terms: 
\begin{align}
\text{ Energy } 
& = \underbrace{\text{Fitting residual 
    on the union of core regions
    in } \mathcal{A}_k}_{\mathcal{E}_1} + \underbrace{\text{ Stability of the identified coefficient}}_{\mathcal{E}_2}. 
\label{e:energy} 
\end{align}
The first term $\mathcal{E}_1$ is the fitting residual at the frequencies on the union of the core regions for the features in $\mathcal{A}_k$:
\begin{equation}
    \mathcal{E}_1 
     := \frac{\|{\mathcal{S}(\bF)_{\mathcal{V}({\mathcal{A}_k})}} {  {\bc_{\mathcal{V}({\mathcal{A}_k})} }} -  {\mathcal{S}(\bb)_{\mathcal{V}({\mathcal{A}_k})}} \|_2}{\|{\mathcal{S}(\bb)_{\mathcal{V}({\mathcal{A}_k})}}\|_2}.
    \label{e: support loss 1}
\end{equation}
Here ${\bc_{\mathcal{V}({\mathcal{A}_k})} }$ denotes the coefficients recovered using  an error-normalized  feature matrix  $\widetilde{\cS(\bF)_{\mathcal{V}(u_t)}}$ and dynamic variable $\widetilde{\cS(\bb)_{\mathcal{V}(u_t)}}$ defined in \eqref{e: least square fit eqn} in Subsection \ref{ss: error normalization}. 
The union of core regions for the features in $\mathcal{A}_k$ is
{
\[
\mathcal{V}(\mathcal{A}_k) = \bigcup_{l \in \mathcal{A}_k} \mathcal{V}(g_l). 
\]} 
While each feature gives different core regions (high response regions), by using the union $\mathcal{V}(\mathcal{A}_k)$, 
we consider fitting on the high response regions for all the features in $\mathcal{A}_k$

The second term in \eqref{e:energy} measures the stability of the identified  coefficients.  For each feature $g_l$ such that $l\in \mathcal{A}_k$, we compute the coefficients $\bc_{\mathcal{V}(g_l)}$ as in \eqref{e: coefficient on u_t} while the core region $\mathcal{V}(u_t)$ is replaced by $\mathcal{V}(g_l)$.   We compare the normalized  distance among all  coefficient vectors computed on the core region of each feature in $\mathcal{A}_k$, and define the follwing stability term:
\begin{equation}
 \mathcal{E}_2 : =  
    \frac{1}{\big|\mathcal{A}_{k} \big|^2} 
    \frac{1}{\big\|\bc_{\mathcal{V}(u_t)} \big\|_2}
    \sum_{l, l' \in \mathcal{A}_k, l \neq l'}  
\big\|\bc_{\mathcal{V}(g_l)} - \bc_{\mathcal{V}(g_{l'})}\big\|_2.
    \label{e: support loss 2}
\end{equation}
If fitting on the core region of one feature gives a very different coefficient vector from fitting on the core region of another feature, this stability term in \eqref{e: support loss 2} will be larger than when the  coefficients computed on the core regions of different features are similar.  This term measures the stability of coefficient computation on the core regions of different features in $\mathcal{A}_k$. 

We find the optimal sparsity $k^*$ and the associated support set $\mathcal{A}_{k^*}$ by minimizing the energy $\mathcal{E}_1 + \mathcal{E}_2$:
\begin{equation}
k^*= \argmin_k 
\left\{ \frac{\|{\mathcal{S}(\bF)_{\mathcal{V}({\mathcal{A}_k})}} { \bc_{\mathcal{V}({\mathcal{A}_k})}} -  {\mathcal{S}(\bb)_{\mathcal{V}({\mathcal{A}_k})}} \|_2}{\|{\mathcal{S}(\bb)_{\mathcal{V}({\mathcal{A}_k})}}\|_2}
+      
     \frac{1}{\big|\mathcal{A}_{k} \big|^2} 
    \frac{1}{\big\|\bc_{\mathcal{V}(u_t)} \big\|_2}
    \sum_{l, l' \in \mathcal{A}_k, l \neq l'}  
\big\|\bc_{\mathcal{V}(g_l)} - \bc_{\mathcal{V}(g_{l'})}\big\|_2\right\}.
    \label{e:energyCoreRegion}
\end{equation}
The energy summing \eqref{e: support loss 1} and \eqref{e: support loss 2}  measures how meaningful the recovered support $\mathcal{A}_k$ is in terms of 
the best fitting on the union of core regions and the stability of coefficient identification.  

Finally, we compute the best coefficient vector $\bc^*$ by
solving a least square problem on a properly selected core region, that gives the smallest residual  among the features in $\mathcal{A}_{k^*}$ and $u_t$.
For each core region of each features in $\mathcal{A}_{k^*}$ and $u_t$, residual error is computed.  
Let $\mathcal{V}^*$ be the best core region of features among all features in $\mathcal{A}_{k^*}$ and $u_t$, which gives the minimum residual:
\begin{equation}
    \mathcal{V}^* = \argmin_{V \in \big\{\mathcal{V}(u_t), \mathcal{V}({g_l}) : l \in \mathcal{A}_{k^*} \big\} }\frac{\|\mathcal{S}(\bF)_{V} \bc_{V} - \mathcal{S}(\bb)_{V}\|_2}{\|\mathcal{S}(\bb)_{V}\|_2}.
    \label{e: coeff selection}
\end{equation}
The core regions for each feature in $\mathcal{A}_{k^*}$ and $u_t$ are  used to find the minimum residual in \eqref{e: coeff selection} such that $V \in \big\{\mathcal{V}(u_t), \mathcal{V}({g_l}) : l \in \mathcal{A}_{k^*} \big\}$. In our experiments, using the union $ \mathcal{V}(\mathcal{A}_{k^*})$ does not give good results. When  fewer number of features are used for the core region, it  gives better coefficient recovery. In \eqref{e: coeff selection}, we experiment on the core regions for each feature, and pick one feature to define the core region.  

With $\mathcal{V}^*$,  \textit{the identified coefficient vector} $\bc^*$ is given by 
\begin{equation}
    \bc^* = \rm{ LeastSquare}( \widetilde{\cS(\bF)_{\mathcal{V}^*}}, \widetilde{\cS(\bb)_{\mathcal{V}^*}})
    \label{e: c optimal}
\end{equation}
{using error-normalized feature matrix $\widetilde{{\cS(\bF)_{\mathcal{V}^*}}}$ and dynamic variable $\widetilde{{\cS(\bb)_{\mathcal{V}^*}}}$  as defined in \eqref{e: least square fit eqn}.}

\section{Numerical Implementation Details}
\label{s: numerical implementation details}
We present implementation details in this section.  First, we propose a method to extend and augment the data when the given data are not periodic. Secondly, we discuss the details of how the threshold is computed for the core regions of features.  Thirdly, we review the details of error-normalization on how it is applied in the frequency domain.  

\subsection{Domain extension for different boundary conditions}
\label{ss: boundary extension}
Applying the Fourier Transform on $u$ requires $u$ to be periodic along both $t$ and $x$ directions. When $u(x,t)$ does not satisfy this periodic condition, we extend it to a periodic function to compute Fourier coefficients. 
 
We introduce two transformation operators $\mathcal{H}_t$ and $\mathcal{G}_t$ defined as follows:
$${\mathcal{G}_{t}(f_l(u(x,t)))} = 
\begin{cases}
f_l(u(x,t)) & 0 \leq x \leq X, 0 \leq t \leq T,\\
-f_l(u(x,2T - t)) & 0 \leq x \leq X, T \leq t \leq 2T, \\
\end{cases}
$$ 
and
$${\mathcal{H}_t(u(x,t))} = 
\begin{cases}
u(x,t) & 0 \leq x \leq X, 0 \leq t \leq T,\\
u(x,2T - t) & 0 \leq x \leq X, T \leq t \leq 2T, \\
\end{cases}$$
where $u(x,t)$ is a continues function for $x \in [0, X]$ and $t \in [0, T]$. The subscript $t$ denotes that we are extending $u(x,t)$ along $t$. One can easily check that both extended functions ${\mathcal{G}_t(u(x,t))}$ and ${\mathcal{H}_t(u(x,t))}$ are periodic along $t$. Then the equation \eqref{eq.pde} is converted to an alternative form 
\begin{align}
	\frac{\partial \mathcal{H}_t(u) }{\partial t}(x,t)=\sum_{l=1}^L c_l \frac{\partial^{\alpha_l} \mathcal{G}_{t}(f_l(u))}{\partial x^{\alpha_l}},
 \label{e: pde alternative}
\end{align}
where $u_t$ is computed with $\mathcal{H}_t(U_{noise, i}^n)$ instead of $U_{noise, i}^n$ and 
features on the right-hand side are computed using $\mathcal{G}_t(f_l(U_{noise,i}^n))$.
For simplification of notation, we present the paper with  $U_{noise, i}^n$. However,  when the given data do not satisfy the period boundary condition, we consistently use the extended boundary equation \eqref{e: pde alternative} instead of \eqref{eq.pde}.

\subsection{Threshold computation for the core region of features}
\label{s: choice of beta}

We present how the  threshold $\beta_{u_t}$ 
used in Subsection \ref{ss:meaningfulRegion} is chosen to determine the core region of feature for $u_t$.  
In  Subsection \ref{ss:meaningfulRegion}, we  partition the meaningful data region $\Lambda$ further into 
the core region (high response region of the feature), 
and the noise region (low response region of the feature).  The partition is based on the threshold $\beta_{u_t}$.
\begin{enumerate}
\item First, we collect absolute values of the smoothed responses for   $u_t$  in $\Lambda$ in a new set denoted by $\mathcal{B}$.  We take the absolute values of both the real and imaginary parts 
for each index $h=\mathcal{H}(\xi_x,\xi_t)$. 
Here the frequency index $h = \mathcal{H}(\xi_x,\xi_t)$ is in the meaningful data region $\Lambda$ such that $\xi_x =1,\ldots, a_x^*$ and $\xi_t =1,\ldots, a_t^*$.

    \item We partition the range of 
    frequency responses in $\mathcal{B}$  into a fixed number of bins $N_{\mathcal{B}} = 300$. 
    The partition is on an equally spaced grid.
We denote the index of each bin by ${\theta}=1,...,N_{\mathcal{B}}$. 
Let
$b_{\theta}$ represent the number of Fourier responses located in the ${\theta}$th bin. 
In other words, each   $b_{\theta}$ counts the number of elements of $\mathcal{B}$ with values in $[b_{\theta}^{\rm left}, b_{\theta}^{\rm right}]$, where $b_{\theta}^{\rm left}, b_{\theta}^{\rm right}$ denote the lower and upper bound of the responses in the ${\theta}$th bin. 
    \item We apply a two-piece linear fit on the cumulative sums of these bins, denoted by $B(k) = \sum_{{\theta}=1}^k b_{\theta}$ for $k = 1,2,..., N_{\mathcal{B}}$. 
     The threshold $\beta_{u_t}$ is chosen as $b_{{\theta}^*+1}^{\rm left}$ 
    where ${\theta}^*$ is determined by the minimizer of the sum of two linear fitting residuals: \[
    \beta_{u_t} = b_{{\theta}^*+1}^{\rm left}\]
    where 
    \[ {\theta}^* = \argmin_{{\theta}}\bigg\{ \sum_{{\theta} = 2}^{k} \left(\frac{B(k)-B(1)}{k-1}{\theta} + B(1) \right)+ 
    \sum_{{\theta} = k+1}^{N_{\mathcal{B}}-1} \left(\frac{B(N_{\mathcal{B}})-B(k+1)}{N_{\mathcal{B}} - k -1}{\theta} + B(k+1) \right)
    \bigg\}.
\]
\end{enumerate}
For each feature $g_l$, the core region of feature is  computed with the threshold $\beta_{g_l}$. We choose the threshold $\beta_{g_l}$ in a similar way to  $\beta_{u_t}$.

\subsection{Error-normalization on the core region of features}
\label{ss: error normalization}

Whenever the coefficient vector is computed by least-squares, e.g., \eqref{e: coefficient on u_t}, \eqref{e:contribution score}, \eqref{e: support loss 1}, \eqref{e: support loss 2}, \eqref{e:energyCoreRegion}, \eqref{e: coeff selection}, \eqref{e: c optimal} error-normalization \cite{tang2023weakident} is used.  The motivation of error-normalization is to unify the effect of noise among different feature terms in the library.  Here we illustrate how error-normalization is implemented in the frequency domain.
 Using the noise model of the given data \eqref{e:u+eps},  each Fourier feature $g_l$  has the following expression
\begin{equation}\label{e:errorIN}
     \left( \frac{2\pi \xi_x}{X} \sqrt{-1} \right)^{\alpha}\sum_{i=0}^{\mbNx-1} \sum_{n=0}^{\mbNt-1} \Biggl\{e^{\displaystyle -\left(\frac{2\pi i}{\mathbb{N}_x}\xi_x + \frac{2 \pi n}{\mathbb{N}_t}\xi_t \right)\sqrt{-1}}
    \left(U_i^n +\epsilon_i^n \right)^{\beta} \Biggr\}
\end{equation}
at each $h=\mathcal{H}(\xi_x,\xi_t)$. 
We consider the maximum of $\epsilon_i^n$ as $\epsilon$, and expand \eqref{e:errorIN} in terms of  $\epsilon$, then the leading coefficient error of each feature $g_l$ becomes
\begin{equation}
\left( \frac{2\pi \xi_x}{X} \sqrt{-1} \right)^{\alpha_l}\sum_{i=0}^{\mbNx-1} \sum_{n=0}^{\mbNt-1} \Biggl\{e^{\displaystyle -\left(\frac{2\pi i}{\mathbb{N}_x}\xi_x + \frac{2 \pi n}{\mathbb{N}_t}\xi_t \right)\sqrt{-1}}
    \beta_l \left(U_i^n \right)^{\beta_l-1} \Biggr\}
    \label{e:leadingerrcoeff}
\end{equation}
when $\beta_l >1$.  Notice that this term is  the same as the feature terms with one lower $\beta_l$ degree, i.e., for feature term $g_l=\frac{\partial^{\alpha_l} u^{\beta_l}}{\partial x^{\alpha_l}}$ the leading coefficient error \eqref{e:errorIN} is the Fourier feature of $\frac{\partial^{\alpha_l} u^{\beta_l-1}}{\partial x^{\alpha_l}}$, the feature term of monomial of degree $\beta_l - 1$ and same derivative order $\alpha_l$. We define this feature index as $l^*$.
For $\beta_l=1$, the leading coefficient error is the feature itself deduced from \eqref{e:errorIN}. 

In the discrete form, we use the denoised form Fourier transform \eqref{e:denoisingS} and take the average for $h \in \mathcal{V}(u_t)$, and refer to as $s(l)$:
\begin{equation}
\begin{aligned}
   s(l) = \begin{cases}
       \beta_l \text{Avg}_{h \in {\mathcal{V}(u_t)}} | \overline{\mathcal{S}(\bF)_{\mathcal{V}(u_t)}}|_{h, l^*},    &   \beta_l > 1, \text{ with index } l^* \text{ representing }  (\alpha_l, \beta_l - 1) \\
    \text{Avg}_{h \in {\mathcal{V}(u_t)}} |\overline{\mathcal{S}(\bF)_{\mathcal{V}(u_t)}}|_{h, l},
    &  \beta_l = 1,  \text{ with index } l \text{ being }  (\alpha_l, \beta_l)
    \end{cases}
\end{aligned}
\label{e: s l }
\end{equation}
where $\overline{|\mathcal{S}(\bF)|}$ represents the magnitude of the smoothed feature matrix with a vertical stacking of the real and imaginative features \eqref{e: Fc = b on A u_t with phi}. 
Here ${\rm Avg}_{h \in \mathcal{V}(u_t)}$ denotes the average operation  over $h \in \mathcal{V}(u_t)$, and $l$, and $l^*$ denote the column index of a feature with both derivative order $\alpha_l$ and polynomial degree $\beta_l$, and $\beta_l - 1$ (as defined in (\ref{e:errorIN}) and (\ref{e:leadingerrcoeff})), respectively. 
Similarly, the scale constant for $\overline{\mathcal{S}({\bb})_{\mathcal{V}(u_t)}}$ is 
\begin{equation}
 s(b) = \text{Avg}_{h \in \mathcal{V}(u_t)} \overline{|\mathcal{S}(\bb)_{\mathcal{V}(u_t)}|}_{h}.
 \label{e: s b}
\end{equation}
Using these scales, the error-normalization is to use  \begin{equation*}
    \widetilde{\mathcal{S}({\bF})_{\mathbb{A}}} = \overline{\mathcal{S}({\bF})_{\mathbb{A}} }\cdot {\rm diag}\left(\frac{1}{s(1)},...,\frac{1}{s(L)}\right), \; \quad \text{ and } \; \quad 
    \widetilde{\mathcal{S}({\bb})_{\mathbb{A}}} = \overline{\mathcal{S}( {\bb})_{\mathbb{A}}} \cdot \frac{1}{{s(b)}}
\end{equation*}
to construct the rescaled system:
\begin{equation}
     \widetilde{\mathcal{S}({\bF})_{\mathbb{A}}}  \widetilde{\bc} = 
     \widetilde{\mathcal{S}({\bb})_{\mathbb{A}}}. 
     \label{e: least square fit eqn}
\end{equation}
Note that this is for a general region $\mathbb{A}$ which can be the core region of feature $u_t$ of $g_l$, but the same scale terms $s(l)$ in \eqref{e: s l } and $s(b)$ in \eqref{e: s b} are consistently used. 
Once we compute the coefficient vector $ \widetilde{\bc}$, it is rescaled to the coefficient ${\bc}$ by  
\begin{equation}
     {\bc}=  
      \widetilde{\bc} \cdot 
    {\rm diag}\left(\frac{s(b)}{s(1)},...,\frac{s(b )}{s(L )}\right).
    \label{e: rescaled vector}
\end{equation}

\section{Numerical experiments}
\label{s: numerical experiments}
In this section, we present numerical experiments. The noise $\epsilon_i^n$ for $i = 0,...,\mathbb{N}_x-1$, $n = 0,...,\mathbb{N}_t-1$,  follows a Gaussian distribution with mean zero and variance $\sigma_{\rm Noise}^2$. For the noise-to-signal ratio (NSR), we use the following definition:
$$
\sigma_{\rm NSR} = \frac{ \sigma_{\rm Noise}}{\sqrt{\frac{1}{\mathbb{N}_t\mathbb{N}_x}  \sum_{{i},n}|{U}_i^n|^2}}.
$$ 

To measure the accuracy of identification, we use four identification errors in Table \ref {t: identification errors}. 
They are the Relative coefficient error $e_2$, Relative residual error $e_{res}$, True Positive Rate (TPR), and Positive Predictive Value (PPV).  Here $e_2$ measures the accuracy of the coefficient identification compared to the true coefficient; $e_{res}$ shows the relative residual error of fitting the identified differential equation to the given data; 
TPR provides the percentage of how many true features are identified compared to the total count of the true features; PPV provides the percentage of the true features identified as a result among all the identified features.

\begin{table}[t!]
    \centering
    \begin{tabular}{lc|lc}
    \toprule
    Name & Definition & Name & Definition \\
    \midrule
    $e_2$   &  \parbox{4cm}{\[ \frac{\|\bc-\bc_{\rm true}\|_2}{\|\bc\|_2} \]}  &
    TPR & \parbox{7cm}{\[ \frac{|\{l: c_{\rm pred}(l) \neq 0, c_{\rm true}(l) \neq 0 \}|}{|\{l: c_{\rm true}(l) \neq 0|}\]}\\
    $e_{\rm  res}$   &  \parbox{4cm}{\[\frac{\|\bF \bc - \bb \|_2}{\|\bb\|_2} \] }  &
    PPV & \parbox{7cm}{\[  \frac{|\{l: c_{\rm pred}(l) \neq 0, c_{\rm true}(l) \neq 0 \}|}{|\{l: c_{\rm pred}(l) \neq 0|}\]}\\
    \bottomrule
    \end{tabular}
    \caption{We use four errors to measure the accuracy of identifying a partial differential equation in this paper:  
    Relative coefficient error ($e_2$), Relative residual error ($e_{\rm res}$), True Positive Rate (TPR), and Positive Predictive Value (PPV). }
    \label{t: identification errors}
\end{table}

We experiment on various differential equations listed in Table \ref{t: equations}. 
The Heat equation \eqref{e: heat}, Transport equation \eqref{e: tran diff}, and Burgers' equation \eqref{e: burgers diff w diff} are simulated using the spectrum method. The KdV \eqref{e: kdv} and KS \eqref{e: ks} equation are simulated using the ETD RK4 method \cite{kassam2005fourth}.  The true features in these equations contains $u_x, (u^2)_x, u_{xx}, u_{xxx},$ and $u_{xxxx}$. A diffusion term is added to the transport and burger equation to stabilize the pattern and increase the complexity of the model. 
These experiments show the robustness of FourierIdent in identifying linear, nonlinear, or higher-order derivative features.

\begin{table}[t!]
    \centering
    \begin{tabular}{p{3cm}cp{6cm}}
    \toprule
    Name & Definition & Simulation paramters \\
    \midrule
    Heat Equation & \parbox{6cm}{\begin{equation} \frac{\partial u}{\partial t} = 0.1 \frac{\partial^2 u}{\partial x^2}    
    \label{e: heat} \end{equation}} 
    & $T = [0, 0.0999]$, $[X_1, X_2] = [0,10] $, $\Delta x = 0.0391$, $\Delta t = 0.003$.\\
    Transport Equation with diffusion & \parbox{6cm}{\begin{equation} \frac{\partial u}{\partial t} = - \frac{\partial u}{\partial x}  + 0.1 \frac{\partial^2 u}{\partial x^2}   
    \label{e: tran diff} \end{equation}}
    & $T = [0, 0.0999]$, $[X_1, X_2] = [0,10] $, $\Delta x = 0.0391$, $\Delta t = 0.003$. \\
        Burger's Equation with diffusion   &  \parbox{6cm}{\begin{equation} \frac{\partial u}{\partial t} = 0.25 \frac{\partial }{\partial x} u^2 + 0.05 \frac{\partial^2 u}{\partial x^2} 
    \label{e: burgers diff w diff} \end{equation}}
    & $T = [0,  0.4995]$, $[X_1, X_2] = [-3.1416,  3.1293]$, $\Delta x = 0.0123$, $\Delta t = 0.001$.\\
    Korteweg-de Vires (KdV) equation  & \parbox{6cm}{\begin{equation} \frac{\partial u}{\partial t} = - 0.5\frac{\partial}{\partial x}u^2 -     \frac{\partial^3 u}{\partial x^3} 
    \label{e: kdv} \end{equation}}
    & $T = [0,  0.0200]$, $[X_1, X_2] = [-3.1416, 3.1416]$, $\Delta x =  0.0157$, $\Delta t = 4 \times 10^{-5}$.\\
    Kuramoto-Sivashinsky (KS)  & 
    \parbox{6cm}{\begin{equation} \frac{\partial u}{\partial t} =- 0.5\frac{\partial}{\partial x}u^2  -     \frac{\partial^2 u}{\partial x^2} 
     - \frac{\partial^4 u}{\partial x^4} 
    \label{e: ks} \end{equation}}
    & $T = [0, 150]$, $[X_1, X_2] = [0,100.53]$, $\Delta x = 0.3927$, $\Delta t = 0.5$.\\
    \bottomrule
    \end{tabular}
    \caption{A list of equations used for experiments  of FourierIdent. For all the equations, we use the maximum power $\beta = 6$, maximum order of derivative $\alpha = 6$, and total number of features $L = 43$ as the dictionary's parameters. }
    \label{t: equations}
\end{table}

\subsection{Workflow of FourierIdent}\label{sec:flow}
We illustrate the procedure of FourierIdent step by step, from constructing a feature matrix to achieving the identified coefficients.

In [Step 1], we obtain the feature matrix  \eqref{e: Fc = b on A u_t with phi} and determine the core regions of features. 
Figure \ref{fig: vis Lambda region} shows an example of the KdV equation \eqref{e: kdv} with $\sigma_{\rm NSR}=0.3$ noise.  (a) shows the meaningful data region $\Lambda$ (red box) in relation to the entire frequency domain, (b) shows the  frequency response of $u_{xxx}$ in $\Lambda$ (zoomed), and (c) shows the core region of $u_{xxx}$ restricted in $\Lambda$.     
Note the significant reduction of the size contributes to the efficiency of the method in \eqref{e: coefficient on u_t}.
\begin{figure}[t!]
    \centering
    \begin{tabular}{ccc}
    (a) Frequency domain & (b) $\mathcal{F}(u_{xxx})$ on $\Lambda$ &  (c) $\mathcal{V}(u_{xxx})$ \\
     \includegraphics[width = 0.25\textwidth]{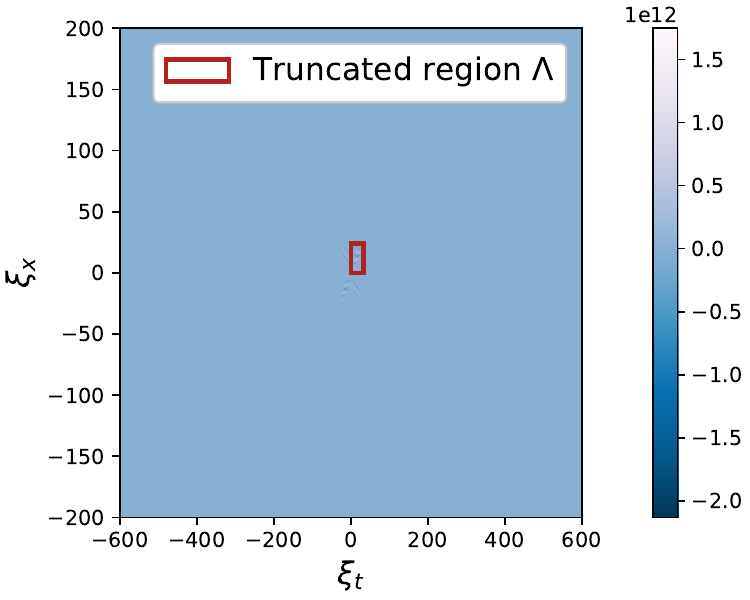}
     & \includegraphics[width = 0.25\textwidth]{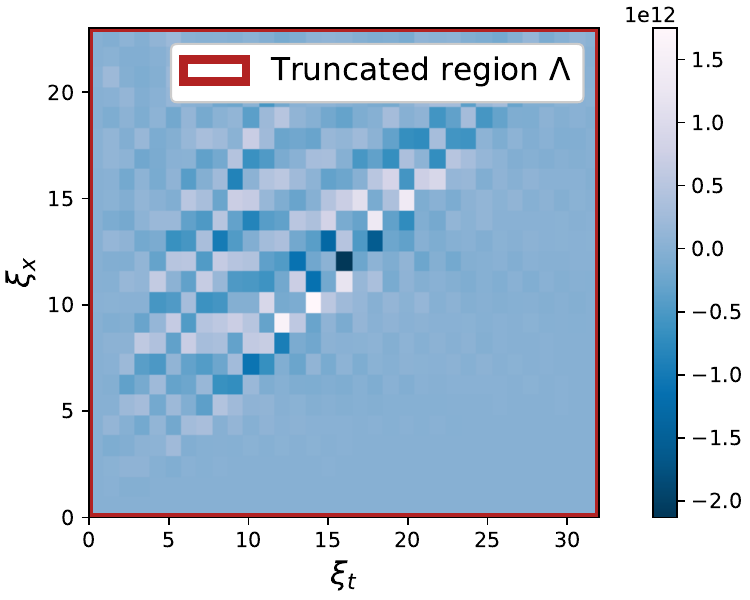}
     &  \includegraphics[width = 0.24\textwidth]{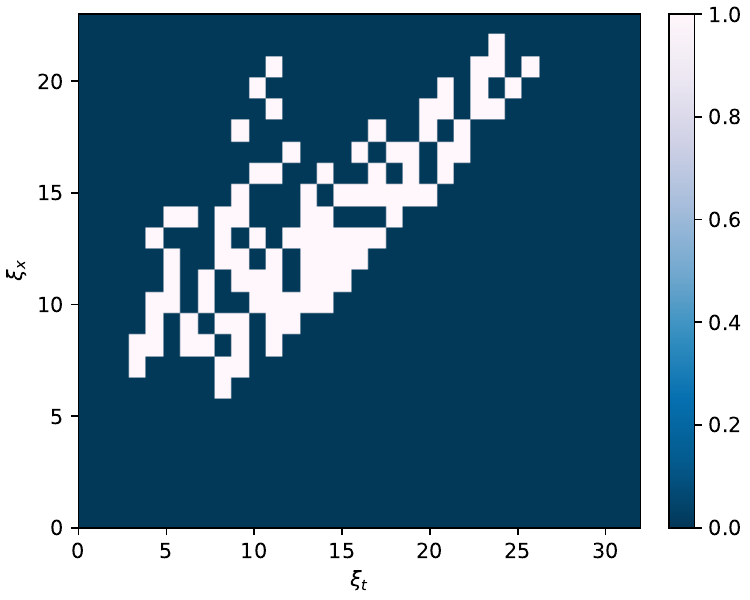}
    \end{tabular}
    \caption{The KdV equation \eqref{e: kdv} with $\sigma_{NSR}=0.3$ as an example. (a) The frequency domain and $\Lambda$ (the red box).  
    (b) Zoom of $\Lambda$ and the frequency response $\mathcal{R}(|\mathcal{F}(u_{xxx})|)$. 
    (c) The white region represents the core region of $u_{xxx}$ further reduced from $\Lambda$.  There is a big  reduction in the size of the region. 
}
    \label{fig: vis Lambda region}
\end{figure}

In [Step 2], for each sparsity level $k=1,\dots, K$, we use Subspace Pursuit to obtain an initial candidate support $\mathcal{A}_k^0$ which may be further reduced to a smaller support through group trimming.  Table \ref{t: example of group trimming} shows the first result of SP, $\mathcal{A}^0_k$, and how group trimming is iteratively applied to get $\mathcal{A}_k$ for the KS equation (\ref{e: ks}). We present the cases for $k\leq 5$, but we computed $\mathcal{A}_k$ until $k=10$. When $k = 1$, the candidate support has one feature $(u^2)_{xxxx}$, and the associated energy \eqref{e:energyCoreRegion} is 0.8071.  When $k = 3$, the correct support is found as the initial support $\mathcal{A}_3^0$, and no more feature is trimmed such that $\mathcal{A}_3= \mathcal{A}_3^0$. When $k = 4$,  the feature $(u^4)_{xxxx}$ is removed during group trimming. When $k=5$, the feature $(u^4)_{xxxx}, (u^2)_{xxx}$ are removed during group trimming. The same support was obtained after group trimming when $3 < k \leq 10$. 
In the end, optimal $k^*$ is given from 3, and the associated support is chosen to be $\mathcal{A}^* = \{ u_{xx}, u_{xxxx}, (u^2)_{x} \}$.

\begin{table}[t!]
    \centering
    \begin{tabular}{lp{10cm} c}
    \toprule
    Sparsity    &  Support 
    &  $k^*$ energy \eqref{e:energyCoreRegion} \\
    \midrule
    SP(1) & $\mathcal{A}_1^0 = \{ (u^2)_{xxxx}\} = \mathcal{A}_1$  & 0.8071 \\ \\
    SP(2) & $\mathcal{A}_2^0 = \{ (u^2)_{xxxx}, (u^2)_{xx}\} = \mathcal{A}_2$
    & 4.1957 \\ \\
    SP(3) & $\mathcal{A}_3^0 = \{ u_{xx}, u_{xxxx}, (u^2)_{x}\} = \mathcal{A}_3$ & 
   \textbf{0.1975} \\ \\
    SP(4) &   $\mathcal{A}_4^0 = \{ u_{xx}, u_{xxxx}, (u^2)_{x}, (u^4)_{xxxx}\}$ & \\ 
     \hspace{1cm} Score \eqref{e:contribution score} & $s(u_{xx})= 1,  s(u_{xxxx})= 0.59,  s((u^2)_{x})= 0.63$, \sout{$s((u^4)_{xxxx})= 0.01$}  & \\
     \hspace{1cm} Updated & $\mathcal{A}_4^1 = \{ u_{xx}, u_{xxxx}, (u^2)_{x}\} = \mathcal{A}_4$ & \textbf{0.1975}\\ \\
    SP(5) & 
    $\mathcal{A}_5^0 = \{u_{xx}, u_{xxxx}, (u^2)_{x}(u^2)_{xxx}, (u^4)_{xxxx}\}$ &\\
    \hspace{1cm} Score \eqref{e:contribution score} & $s(u_{xx})= 1,  s(u_{xxxx})= 0.59,  s((u^2)_{x})= 0.65$, &\\
    & \sout{$s((u^4)_{xxxx})= 0.04, s( (u^2)_x)=0.02$} & \\
    \hspace{1cm} Updated & $ \mathcal{A}_5^1 = \{u_{xx},u_{xxxx}, (u^2)_{x}\}= \mathcal{A}_5$
    & \textbf{0.1975} \\   
   \midrule
    \end{tabular}
        \caption{The KS equation in \eqref{e: ks} with $\sigma_{\rm NSR}=0.5$ as an example. 
        [Step 2] SP and group trimming. For each sparsity level $k$, SP($k$)  represents Subspace Pursuit applied on ${\cS(\bF)_{\mathcal{V}(u_t)}}^{\dagger}$, we show the initial support and the converged support $\mathcal{A}_k$ after the group trimming.  While no features are trimmed when $k \leq 3$, for $k>3$ with the group trimming, it converged in one step.  Note the same features $\mathcal{A}^*$ are found from $k=3$ with the minimum energy \eqref{e:energyCoreRegion}.} 
    \label{t: example of group trimming}
\end{table}

In [Step 3], we pick the coefficient vector with the minimum energy based on the core region of features \eqref{e:energyCoreRegion}.
Note that these coefficient energies are associated with each individual core region, and Table \ref{t: example of coefficient refine} presents this computation from the features of $\mathcal{A}^*$ and $u_t$, i.e., $\mathcal{V}(u_t), \mathcal{V}(u_{xx}), \mathcal{V}(u_{xxxx}), \mathcal{V}((u^2)_x)$, for the KS equation (\ref{e: ks}). 
The coefficient values are similar to each other, thanks to the consistency term \eqref{e: support loss 2} in \eqref{e:energyCoreRegion}, and we choose the best coefficient vector.  In this example, $\mathcal{V}^* = \mathcal{V}(u_{xxxx})$ gives the lowest coefficient energy. The final output of this identification example is   $u_t = - 0.9701 u_{xx} - 0.9916 u_{xxxx} - 0.4785 (u^2)_{x}$ and the true equation is  $u_t = -  u_{xx} - u_{xxxx} - 0.5 (u^2)_{x}$ in \eqref{e: ks}.

\begin{table}[t!]
    \centering
    \begin{tabular}{ll c}
    \midrule
    Features in $\mathcal{A}_k^*$ &  Coefficients  & Residual error \eqref{e: coeff selection}\\
    \midrule
    Use $\cS(\bF)_{\mathcal{V}(u_t)}$ & $\bc_{\mathcal{V}(u_t)} = [-0.8834, -0.881,-0.4443]$ & 0.2817 \\
    Use $\cS(\bF)_{\mathcal{V}(u_{xx})}$ & $\bc_{\mathcal{V}(u_{xx})} =[-0.8912,-0.8921,-0.4464]$  & 0.2817 \\
    Use $\cS(\bF)_{\mathcal{V}(u_{xxxx})}$ & $\bc_{\mathcal{V}(u_{xxxx})} =[-0.9701,-0.9916,-0.4785]$  & \textbf{0.1468} \\
    Use $\cS(\bF)_{\mathcal{V}((u^2)_{x})}$ & $\bc_{\mathcal{V}((u^2)_{x})} =[-0.9642, -0.9706, -0.4867]$  & 0.1483\\
    \midrule
    \multicolumn{3}{l}{Identified equation becomes $u_t = - 0.9701 u_{xx} - 0.9916 u_{xxxx} - 0.4785 (u^2)_{x}$, using  $\mathcal{V}^* = \mathcal{V}(u_{xxxx})$,} \\
    \end{tabular}
    \caption{The KS equation in \eqref{e: ks} with $\sigma_{\rm NSR}=0.5$. [Step 3] for each features of $\mathcal{A}^* \cup \{u_t\} =\{u_t, u_{xx},u_{xxxx}, (u^2)_{x}\}$ computes the coefficients and find the minimum residual energy \eqref{e: coeff selection} feature.   
The minimum coefficient energy is given by $\mathcal{V}^* = \mathcal{V}(u_{xxxx})$ and the corresponding equation is identified. }
    \label{t: example of coefficient refine}
\end{table}

 \subsection{Effect of the meaningful data region $\Lambda$} 
 
We present the effect of the meaningful data region, using the KdV equation with $\sigma_{\rm NSR}=0.3$ as an example.  Figure \ref{fig: scales} shows the scale of features of $\bF$, $\bF_\Lambda$ ($\bF$ restricted on $\Lambda$), $\mathcal{S}(\bF)$,  $\mathcal{S}(\bF)_{\Lambda}$ and $\bW$ which is the Weak-form feature in the physical domain as in WeakIdent \cite{tang2023weakident}.  The magnitude of the Fourier features have a wider range, and using the meaningful data region $\Lambda$ helps to reduce this range by comparing $\bF$ with  $\bF_\Lambda$ and $\mathcal{S}(\bF)$ with  $\mathcal{S}(\bF)_{\Lambda}$.  With $\Lambda$ restriction, the shape of the graph looks similar to the physical values of Weak form.
In addition, due to the symmetry,  a quarter of data is used, i.e., the frequencies $\xi_x,\xi_t$ with the same magnitude have similar behaviors. 
We take a smaller collection of the frequency modes which reduces the  computational cost.

\begin{figure}[t!]
    \centering
    \includegraphics[width = 0.65\textwidth]{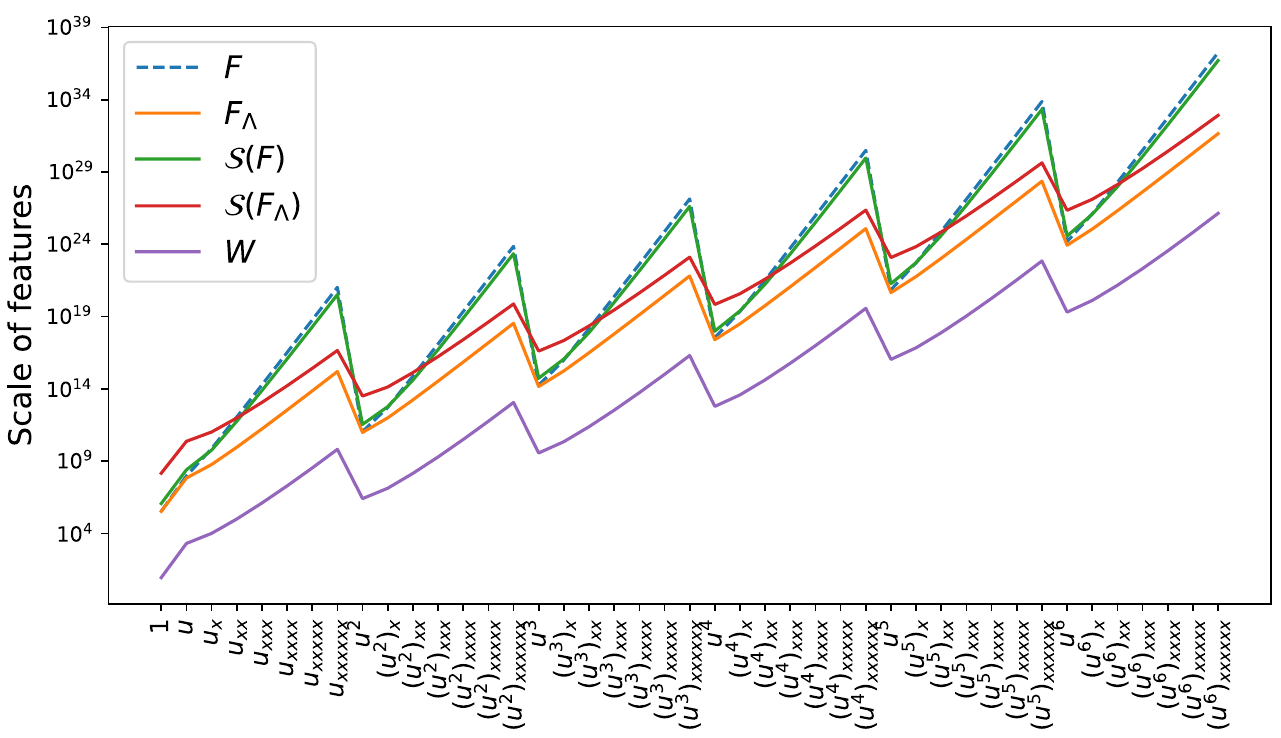} 
     \caption{Effect of the meaningful data region $\Lambda$, illustrated by the KdV equation \eqref{e: kdv}  with $\sigma_{\rm NSR}=0.3$.    The x-axis provides a list of features and the y-axis represents the scale of features in terms of the $\ell_2$-norm of the feature column.  
    We compare $\bF$ with $\bF_\Lambda$, $\mathcal{S}(\bF)$ with  $\mathcal{S}(\bF)_{\Lambda}$, and Fourier features with the Weak-form features in $\bW$   in WeakIdent \cite{tang2023weakident}. A restriction to the meaningful data region $\Lambda$  helps to reduce the range  of Fourier  features  and make the shape of the scale similar to that in  the  weak form. }
    \label{fig: scales}
\end{figure}

In Table \ref{t: use or not use Lambda}, we compare the identification results with or without the restriction to the meaningful data region $\Lambda$. For the KdV equation \eqref{e: kdv} and the KS equation \eqref{e: ks} on clean and noise data $\sigma_{\rm NSR} = 0.3$, we show how the restriction to $\Lambda$ helps to find the correct equation.   Without noise, the restriction to $\Lambda$ makes little differences. When the noise level increases, the results without restriction to $\Lambda$ may give completely wrong results.  

\begin{table}[t!]
    \centering
    \begin{tabular}{l|l|l}
    \toprule
     KdV Eq. \eqref{e: kdv} & true equation & $ u_t = - 1.0 u_{xxx} - 0.5 (u^2)_{x}$\\
     \midrule
    $\sigma_{\rm NSR}=0$    &   with $\Lambda$&
     $u_t = - 0.998 u_{xxx} - 0.499 (u^2)_{x} $\\
     & without $\Lambda$ &
     $u_t = - 0.998 u_{xxx} - 0.499 (u^2)_{x}$\\
     \hline
      $\sigma_{\rm NSR}=0.3$    &   with $\Lambda$ &
     $u_t = - 1.012 u_{xxx} - 0.499 (u^2)_{x}$\\
     & without $\Lambda$ & $u_t = - 151.987 u_{x}$\\
     \midrule
     \midrule
     KS Eq. \eqref{e: ks}& true equation & $u_t = - 1.0 u_{xx} - 1.0 u_{xxxx} - 0.5 (u^2)_{x}$ \\
     \hline
      $\sigma_{\rm NSR}=0$    &   with $\Lambda$&
     $u_t = - 1.0 u_{xx} - 1.0 u_{xxxx} - 0.5 (u^2)_{x} $\\
     & without  $\Lambda$ & $u_t = - 0.999 u_{xx} - 0.999 u_{xxxx} - 0.5 (u^2)_{x} $\\
     \hline
     $\sigma_{\rm NSR}=0.3$    &  with $\Lambda$ & $u_t = - 0.978 u_{xx} - 0.977 u_{xxxx} - 0.49 (u^2)_{x}$\\
     & without  $\Lambda$ &
     $u_t = - 0.035 (u^4)_{x}$\\
     \bottomrule
    \end{tabular}
    \caption{Effect of applying FourierIdent with or without restriction to the meaningful data region $\Lambda$.  With an increased level of noise, the benefit of having $\Lambda$ is clearly demonstrated. } 
    \label{t: use or not use Lambda}
\end{table}

\begin{figure}[t!]
    \centering
    \begin{tabular}{cc}
    (a) $\sigma_{\rm NSR}=0.8$ & (b) $\sigma_{\rm NSR}=0.8$ \\
    \includegraphics[width = 0.37\textwidth]{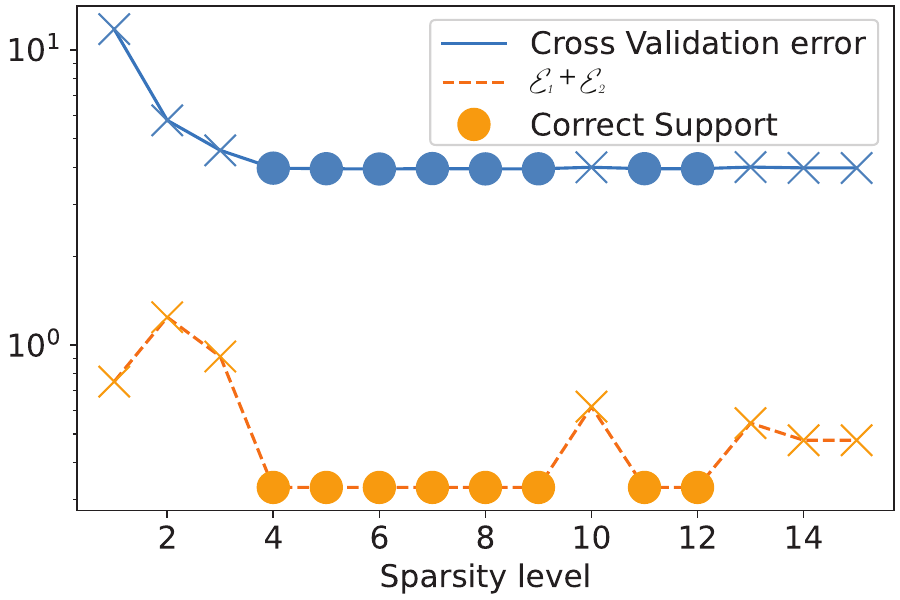}
    &
    \includegraphics[width = 0.37\textwidth]{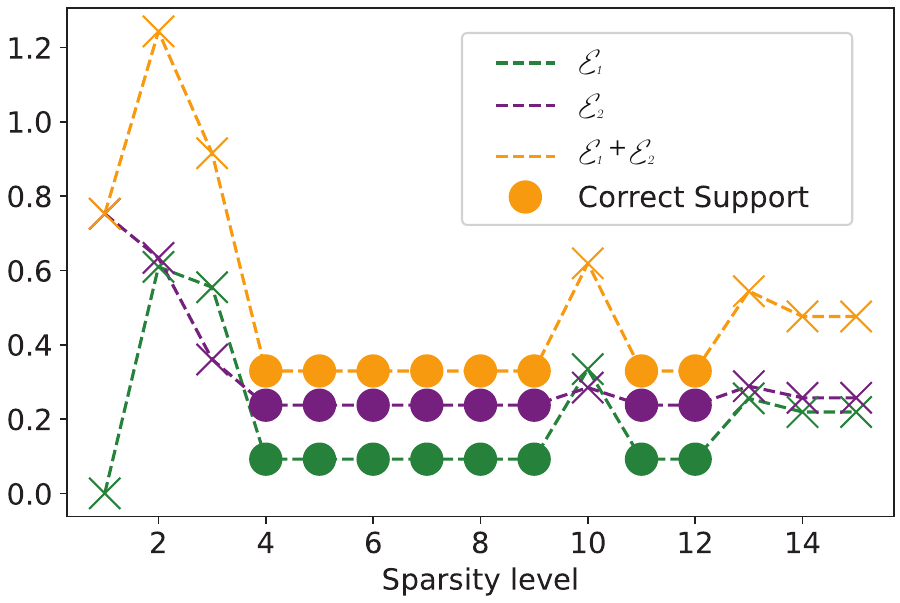}
    \end{tabular}
    \caption{Benefits of using the energy    in \eqref{e:energyCoreRegion} to find the optimal $k^*$.  
    (a) and (b) are both for KS equation with  $\sigma_{\rm NSR}=0.8$. The x-axis represents the initial sparsity level. 
    The dots represent the cases when the trimmed support is the correct one.
    In (a), the y-axis represents the values of Cross-Validation error and the energy in \eqref{e:energyCoreRegion}.
    When the given data are noisy, as in (a), the results with sparsity levels $10,13,14,15$ give rise to  small cross-validation errors but large  energy $\mathcal{E}_1+\mathcal{E}_2$ defined in \eqref{e:energyCoreRegion}. In (b), the green curve shows the fitting residual $\mathcal{E}_1$ in \eqref{e: support loss 1}, and the purple curve shows the stability term $\mathcal{E}_2$ in \eqref{e: support loss 2}. The new energy as the sum of the fitting residual and the stability term is represented by the yellow curve.}
    \label{fig: support energy} 
\end{figure}

\subsection{Understanding the new energy in \eqref{e:energyCoreRegion}}
\label{s: energy analysis experiments}

We introduce a new energy based on the core regions of features in Section \ref{ss: support selection active modes}. In this subsection, we show the effect of the energy in \eqref{e:energyCoreRegion} compared to the Cross-Validation (CV) error used in \cite{he2022robust,tang2023weakident}.  
In Figure \ref{fig: support energy}, we present the result for the KS equation \eqref{e: ks} with $\sigma_{\rm NSR}=0.8$.  In Figure \ref{fig: support energy} (a), the yellow curve shows the new energy \eqref{e:energyCoreRegion} value and the blue curve shows the CV error.  The x-axis shows the initial sparsity as the input of SP, and the sparsity after  group trimming is usually much smaller than the initial sparsity.  The yellow and blue circles indicate the cases when the correct supports are found after group trimming.  
After the sparsity level $k=3$, the CV errors are all similarly low. If we use the CV error, the wrong supports are identified when $k=10, 13, 14, 15$.   The yellow curve using the new energy is more consistent: the  sparsity associated with a low energy corresponds to the correct support.  
Figure \ref{fig: support energy} (b) shows the values of the two terms in the energy \eqref{e:energyCoreRegion}: the green curve shows the fitting residual $\mathcal{E}_1$, 
and the purple curve shows the stability term $\mathcal{E}_2$. The yellow curve shows the sum of them, which is  the same as the yellow curve in Figure \ref{fig: support energy} (a).    While   the residual curve (green curve) is minimized at the wrong sparsity $k=1$, the total energy as the sum of the fitting residual and the stability term is minimized at the correct sparsity and support.

\subsection{FourierIdent experiments with an increasing complexity}
We experiment FourierIdent with an increasing complexity of the initial conditions.   
We use the KdV equation in \eqref{e: kdv}  with different initial conditions (IC) where we can control the complexity:  
\begin{equation}
    u_0 = \sum_{r = 1}^{R} \cos(rx + c_1) + \sin(rx + c_2),
    \label{e:differentmodes}
\end{equation}
where $c_1$ and $c_2$ are two random numbers and $R$ denotes the total number of modes used in $u_0$. The larger the $R$ is, the more complex the given data are in both physical and frequency domains. We use periodic spatial domain such that  $x\in [-3.1416, 3.1293]$ with spatial spacing $\Delta x = 0.123$ and the time domain $t \in [0, 0.02]$ with temporal spacing $\Delta t =   3.9978\times 10^{-5}$.  In Figure \ref{fig:kdv5modes} (a) - (e), we show clean data for the initial conditions with $R=$ 1, 20, 30, 40, and 50 modes. 
It is shown that the pattern in (a) is relatively simple, and pattern in  (e) is the most complex one. We use a simulated solution from each complexity mode from $R = 1$ to $R = 60$.
\begin{figure}[t!]
    \centering
    \begin{tabular}{ccccc}
    (a) 1 mode & (b) 20 modes & (c) 30 modes & (d) 40 modes & (e) 50 modes \\
     \includegraphics[width = 0.18\textwidth]{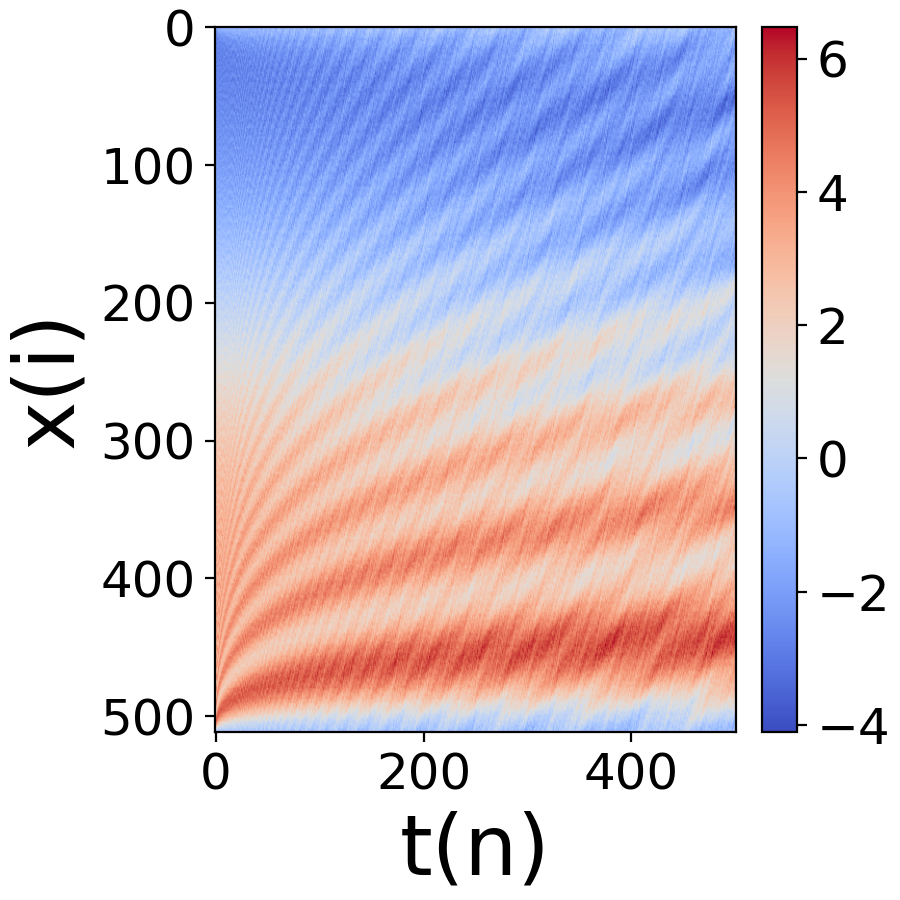}   
     &\includegraphics[width = 0.18\textwidth]{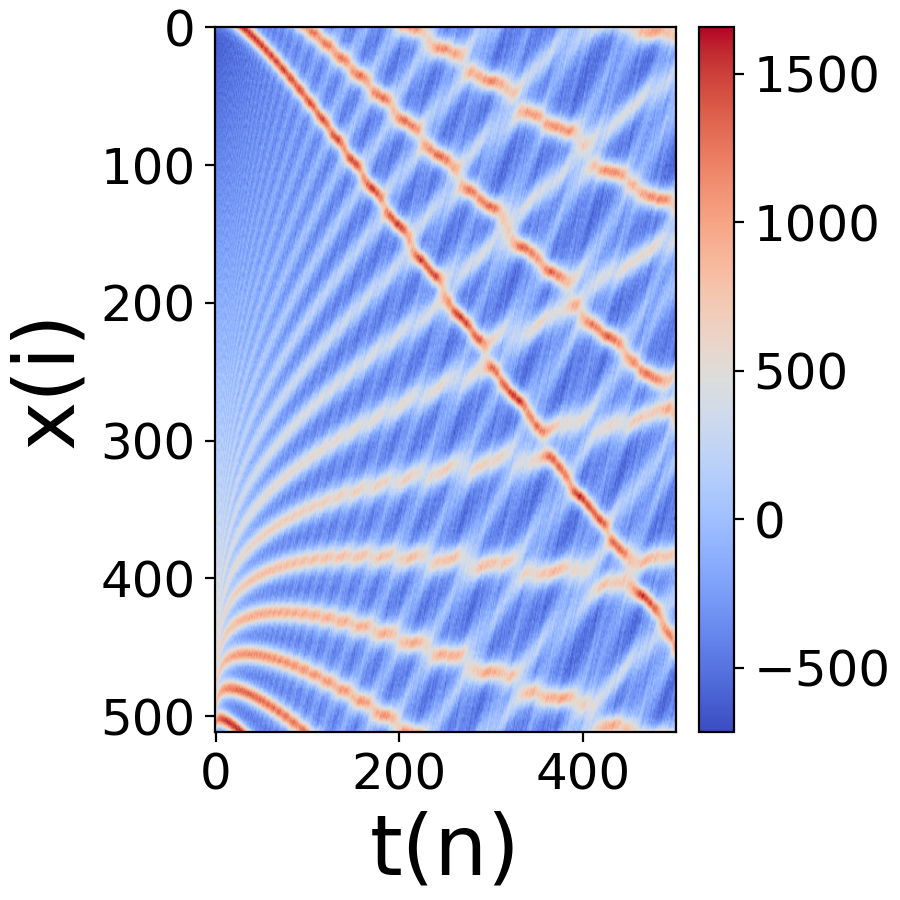} 
     &\includegraphics[width = 0.18\textwidth]{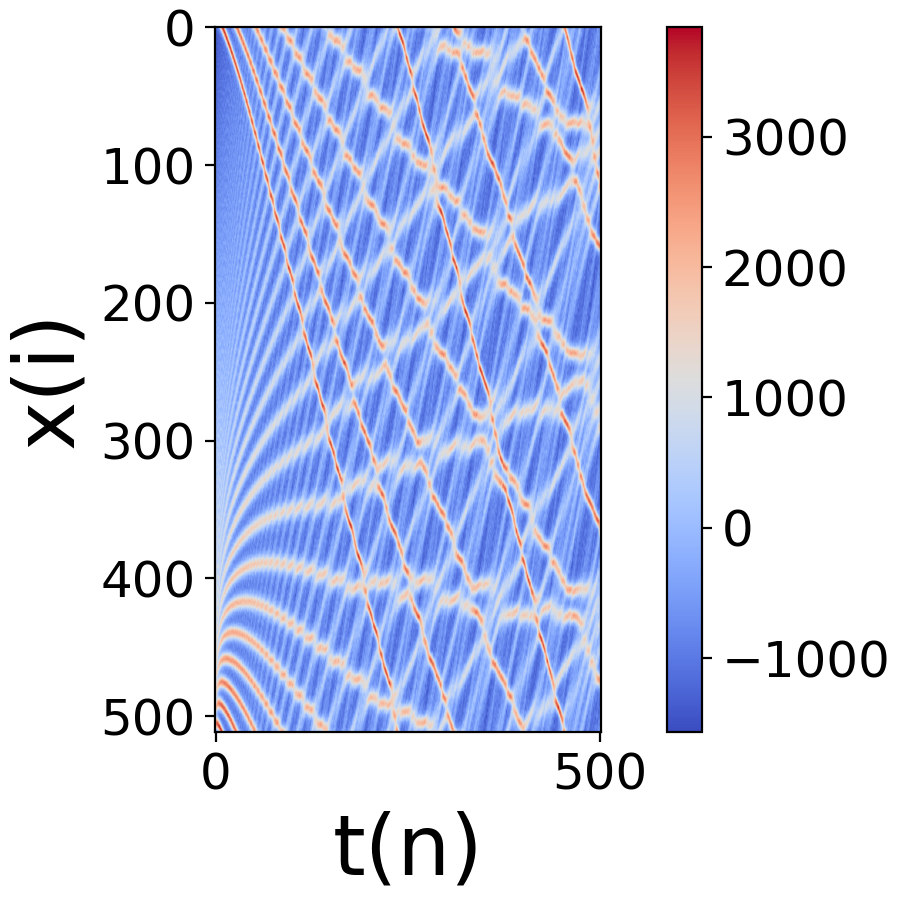} 
     &\includegraphics[width = 0.18\textwidth]{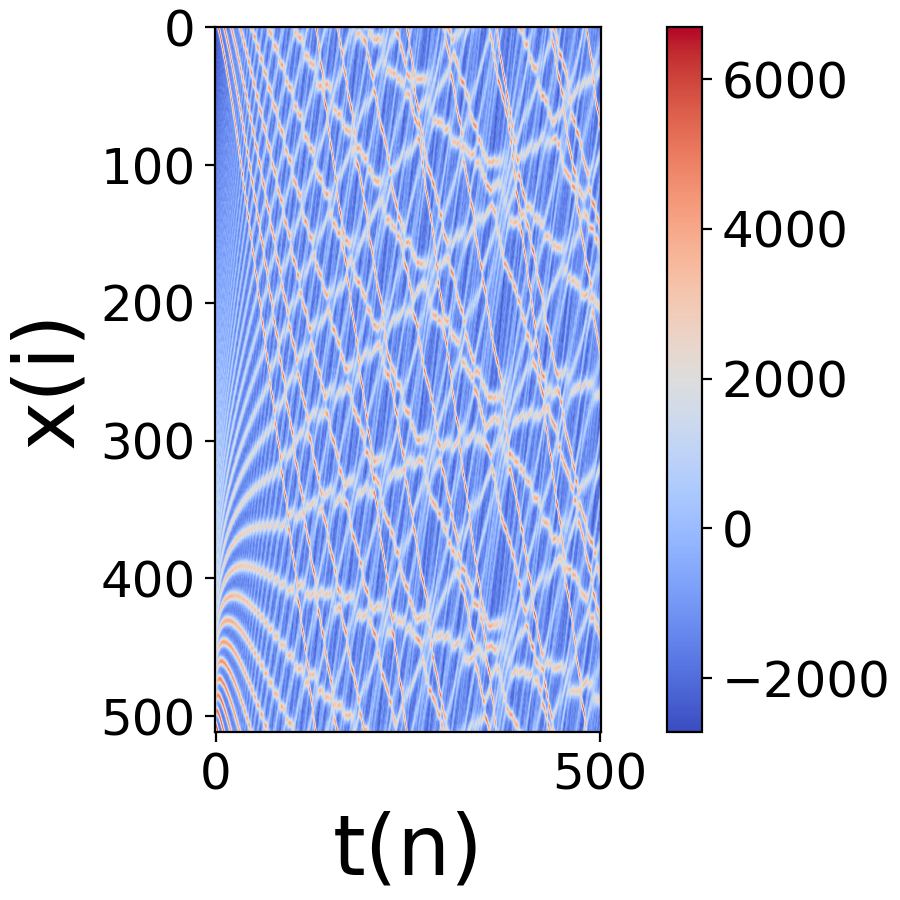}
     &\includegraphics[width = 0.18\textwidth]{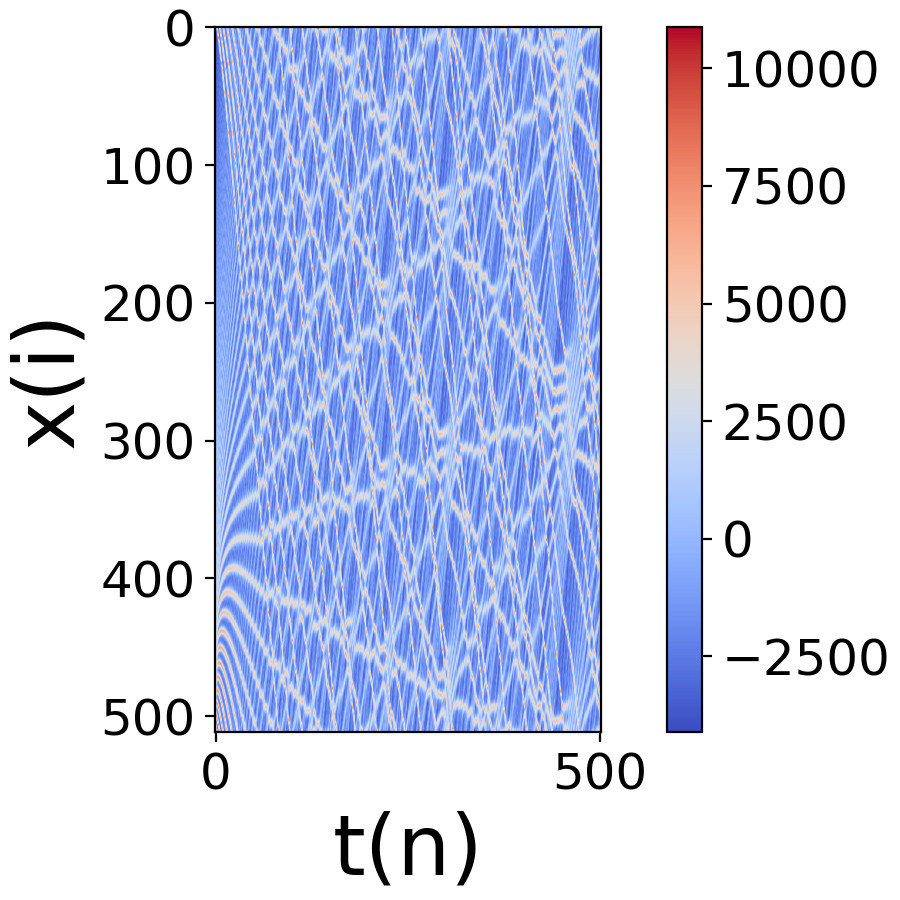} \\
    \end{tabular}
    \begin{tabular}{cccc}
    (f) $\sigma_{\rm NSR}=0.2$ & (g) $\sigma_{\rm NSR}=0.25$ & (h) $\sigma_{\rm NSR}=0.3$ & (i) $\sigma_{\rm NSR}=0.35$  \\
    \includegraphics[width=0.23\textwidth]{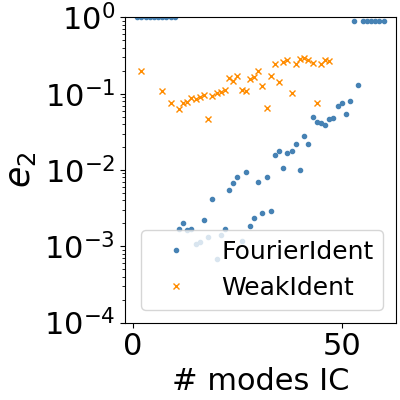} &
     \includegraphics[width=0.23\textwidth]{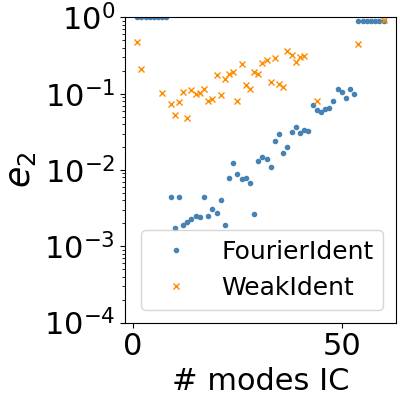} & 
     \includegraphics[width=0.23\textwidth]{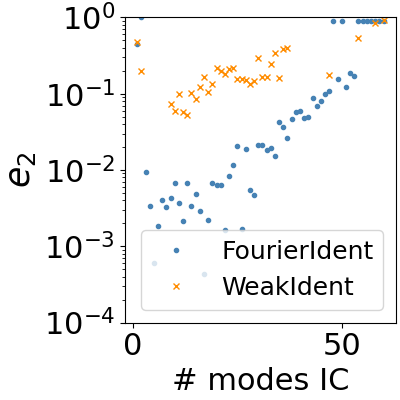} &
     \includegraphics[width=0.23\textwidth]{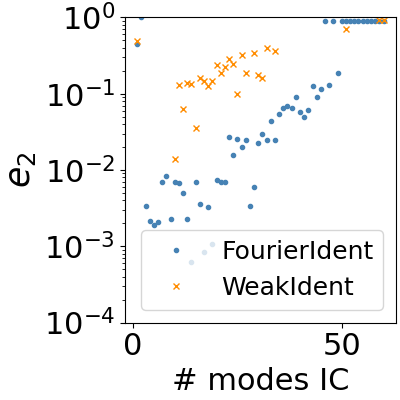} 
       \end{tabular}
\caption{Influence of increasing complexity. (a) - (e) Clean data of the KdV equation \eqref{e: kdv} for the initial condition \eqref{e:differentmodes} with 
$R=$ 1, 20, 30, 40, and 50. From (a) to (e), the patterns become  more complex. 
(f) - (i) show comparison results between FourierIdeant and WeakIdent for different noise levels such that  $\sigma_{\rm NSR} \in \{0.05,0.1,0.15,0.2,0.25,0.3,0.35 \}$. Each dot represents the $e_2$ identification error in each experiment.  When $10 < R <40$, FourierIdent gives rise to more accurate recovery results.
}
\label{fig:kdv5modes}
\end{figure}
Figure \ref{fig:kdv5modes} (f)-(i) show the performance of FourierIdent compared with WeakIdent, for different level of noise such that  $\sigma_{\rm NSR}=
0.2,0.25,0.3,0.35$.   
For each mode $R$ ranging from 1 to 60, we simulate 20 different datasets with different random seeds.
In each graph, each column contains 20 dots for the 20 independent experiments, and the height represents the $e_2$ identification error  of Fourierident (blue) and WeakIdent (yellow). 

This experiment shows that, for  data with different complexity  and with different noise levels,  FourierIdent gives rise to smaller  $e_2$ errors.  Figures (f)-(i) shows that when $K > 10$, FourierIdent (marked in blue) performs better than WeakIdent (marked in yellow), when $\sigma_{NSR}$ is large and the data have more frequency modes.

\subsection{FourierIdent experiment with an increasing time  for data collection
}

The next experiment shows how the increasing of time interval for data collection
can improve the identification result.   Figure \ref{f: limiting behaviors} (a) and (b) show a realization of the given data for the KS equation \eqref{e: ks} with $\sigma_{\rm NSR}=0.5$. (a) is  for $0 < t < 500$ and (b) is for $0 < t < 5,000$.  The data in (b) clearly have more repetition of the pattern. 
Figure \ref{f: limiting behaviors} (c)-(e) show experiments with varying noise levels such that $\sigma_{\rm NSR}=0.01, 0.5,  1.0$.  For each graph,  the x-axis represents the final time $t_{end}$ used for the data, and y-axis shows the $e_2$ error of the identified coefficient. 
In (c), when the noise level is low, both FourierIdent and WeakIdent show good results with the $e_2$ error less than $10^{-3}$.  This is based on one experiment for each noise level.  
As the noise level increases, FourierIdent consistently gives better results in (d) and (e).  
(e) shows an extremely noisy case for $\sigma_{\rm NSR} = 1$ when the given data are highly corrupted by noise. For both (d) and (e), as $t_{end}$ gets bigger, both FourierIdent and WeakIdent give rise to decreasing errors , and the error of FourierIdent is lower.   
The core region is defined by the high frequency responses, thus with more data, 
 the response will be stronger and the core region of feature can separate the noise better. FourierIdent is slightly worse than WeakIdent when $\sigma_{\rm SNR}=0.01$ compared to $\sigma_{\rm SNR}=0.5$ or 1.0. { This shows the effect that 
 FourierIent performs well due to having enough high-frequency responses when the data are exposed to large noise.} 
 In other words, FourierIdent is relatively more effective and robust under noise compared to methods using physical features, which is consistent with our observation in Figure \ref{fig:kdv5modes}.

\begin{figure}[t!]
    \centering
    \begin{tabular}{cc}
    (a) $0\leq t \leq 500$ & (b) $0\leq t \leq 5,000$ \\
    \includegraphics[width = 0.46\textwidth]{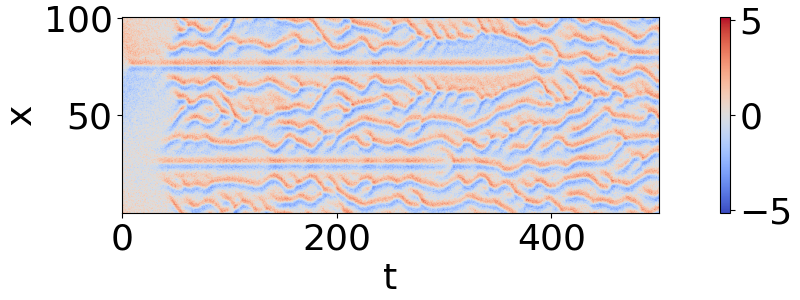}   &
    \includegraphics[width = 0.46\textwidth]{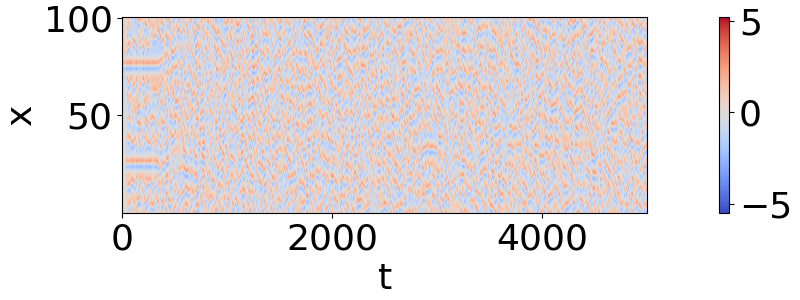} \\
    \end{tabular}
    \begin{tabular}{ccc}
    (c) $\sigma_{\rm NSR} = 0.01$ &  (d) $\sigma_{\rm NSR} = 0.5$ & (e) $\sigma_{\rm NSR} = 1.0$\\
    \includegraphics[width = 0.29\textwidth]{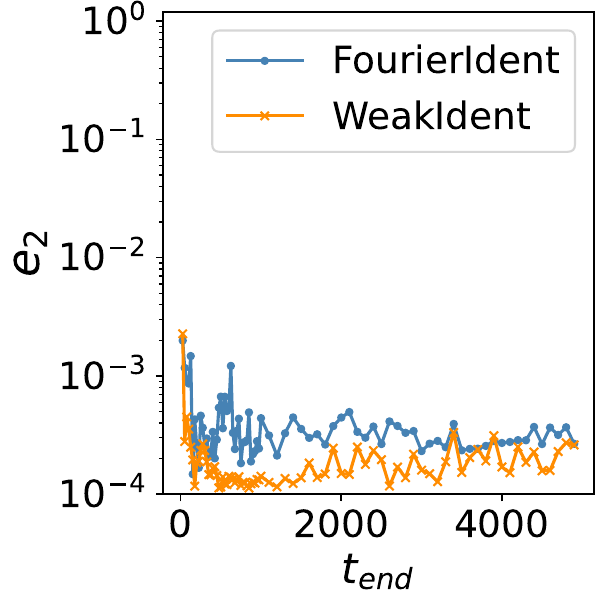}
    & 
    \includegraphics[width = 0.29\textwidth]{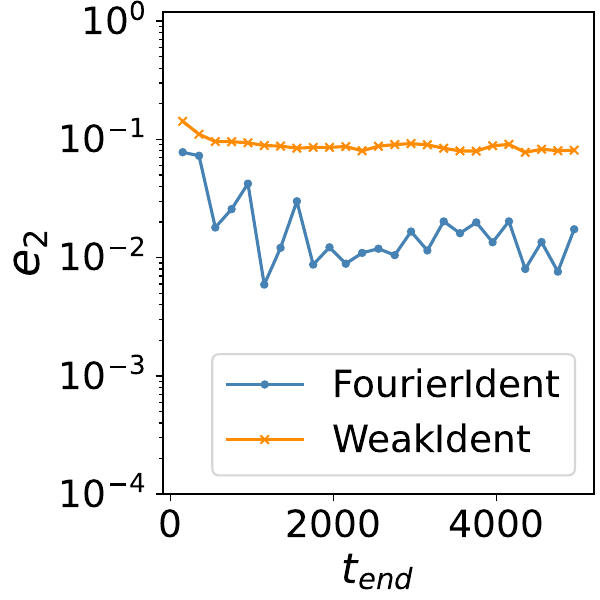} & 
    \includegraphics[width = 0.29\textwidth]{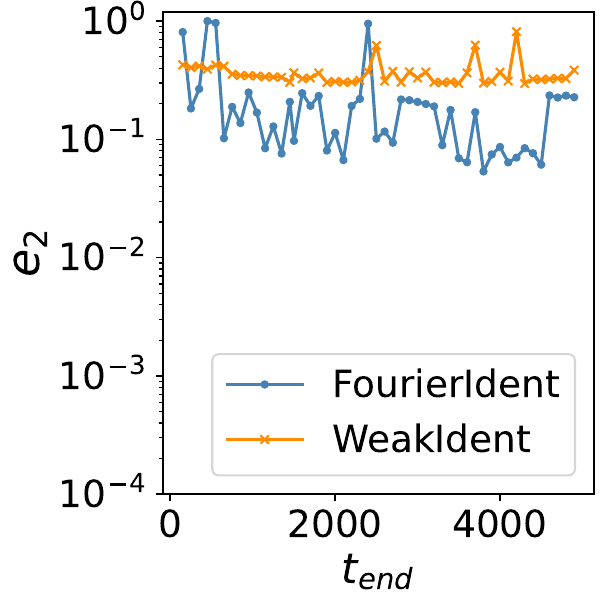} \\
    \end{tabular}
    \caption{Influence of increasing time in data collection. (a) - (b):
    The given data of the KS equation \eqref{e: ks} 
    (a) for $0 < t < 500$, and (b) for $0 < t < 5,000$.  (c)-(e): the $x$-axis represents the final time $t_{end}$  for  data collection, and the $y$-axis shows the $e_2$ error of the identified coefficient, for FourierIdent (blue) and WeakIdent (yellow).  In (c), for low levels of noise  $\sigma_{\rm NSR}=0.01$, both FourierIdent (blue) and WeakIdent (yellow) give small $e_2$ errors. In (d) and (e), when $\sigma_{\rm NSR} = 0.5$ and $1$, as $t_{end}$ gets bigger, both FourierIdent and WeakIdent yield smaller errors, while FourierIdent gives rise to a smaller error than WeakIdent.  }
        \label{f: limiting behaviors}
\end{figure}

\subsection{FourierIdent comparison results}

We test FourierIdent, and compare it with WeakIdent\cite{tang2023weakident} and WSINDy \cite{messenger2021weak}.  
We first present an example showing the identified equations, and then  present more  statistics about the  recovery. 
In Table \ref{f: example output equations}, we present the identified equation and the $e_2$ error for the equations listed in Table \ref{t: equations} with $\sigma_{\rm NSR}=0.3$. 
FourierIdent identifies the correct equation, i.e., the correct support and highly accurate coefficient values, under various scenarios. 

\begin{table}[t!]
\begin{tabular}{lllp{8cm}c}
\toprule
Equ & $\sigma_{\rm NSR}$ & Method & Identified equation & $e_2$ \\
\midrule
\multirow{3}{*}{\eqref{e: heat}} &  \multirow{3}{*}{0.3} & FourierIdent &  $u_t = + 0.1 u_{xx} $ & 0.002536 \\
 &  & WeakIdent & $u_t = + 0.1 u_{xx}$  & 0.003609 \\
 &  & WSINDy & $u_t = 0.098 u_{xx}$ & 0.012311\\
\midrule
\multirow{3}{*}{\eqref{e: tran diff}} & \multirow{3}{*}{0.3} & FourierIdent &  $u_t = - 1.0 u_{x} + 0.099 u_{xx} $ &  0.000719\\
 &  & WeakIdent & $u_t = - 0.997 u_{x} + 0.102 u_{xx}$ &  0.003477 \\
 &  & WSINDy & $u_t = 0.434 + -0.793u_{xxxxxx} -1.003u^2 +0.096(u^2)_{x} +0.470(u^2)_{xxxxxx}$ & 1.017465\\
\midrule
\multirow{3}{*}{\eqref{e: burgers diff w diff}} & \multirow{3}{*}{0.3} & FourierIdent & $u_t = + 0.051 u_{xx} + 0.249 (u^2)_{x}$ &0.006846 \\
 &  & WeakIdent & $u_t = + 0.052 u_{xx} + 0.248 (u^2)_{x} $ & 0.011226 \\
 &  & WSINDy 
 & $u_t = + 0.048 u_{xx} + 0.248 (u^2)_{x}$ & 0.009981\\
\midrule
\multirow{3}{*}{\eqref{e: kdv}} & \multirow{3}{*}{0.3} & FourierIdent & $u_t = - 0.997 u_{xxx} - 0.499 (u^2)_{x}$ & 0.0031048 \\
 &  & WeakIdent & $u_t = - 0.987 u_{xxx} - 0.497 (u^2)_{x}$ & 0.0121109  \\
 &  & WSINDy & $u_t = -0.977 u_{xxx} -0.4967 (u^2)_{x} $ & 0.0203434 \\
 \midrule
\multirow{3}{*}{\eqref{e: ks}} & \multirow{3}{*}{0.3} & FourierIdent & $u_t = - 0.977 u_{xx} - 0.97 u_{xxxx} - 0.488 (u^2)_{x}$ &  0.02663459 \\
 &  & WeakIdent &$u_t = - 0.95 u_{xx} - 0.947 u_{xxxx} - 0.476 (u^2)_{x}$  & 0.051067 \\
 &  & WSINDy & $u_t = -0.9529u_{xx} -0.9493u_{xxxx} -  -0.4779 (u^2)_{x}$ & 0.048433\\
 \bottomrule
\end{tabular}
\caption{Identification results by FourierIdent, WeakIdent\cite{tang2023weakident} and WSINDy \cite{messenger2021weak} for the equations in Table \ref{t: equations} with $\sigma_{\rm NSR}=0.3$ using one realization of the PDE solution.}
\label{f: example output equations}
\end{table}

We further show the recovery statistics for the equations in Table \ref{t: equations}, in Figure \ref{f: comparison heat}, \ref{f: comparison bg}, \ref{f: comparison trans}, \ref{f: comparison kdv}, and \ref{f: comparison ks}. 
We present the results with different noise-to-signal-ratio $\sigma_{\rm NSR}$ using box plots of the identification error for 
FourierIdent, WeakIdent\cite{tang2023weakident} and WSINDy \cite{messenger2021weak}.
For each $\sigma_{\rm NSR}$, the dataset is simulated 20 times with various random seeds of noise.  FourierIdent is compatible with other methods in  identifying a partial differential equation. It is robust to high  levels of noise and capable of dealing with features with high-order derivatives.

\begin{figure}[t!]
    \centering
    \begin{tabular}{cccc}
     (a) & (b) &  (c) & (d) \\
     \includegraphics[width = 0.23\textwidth]{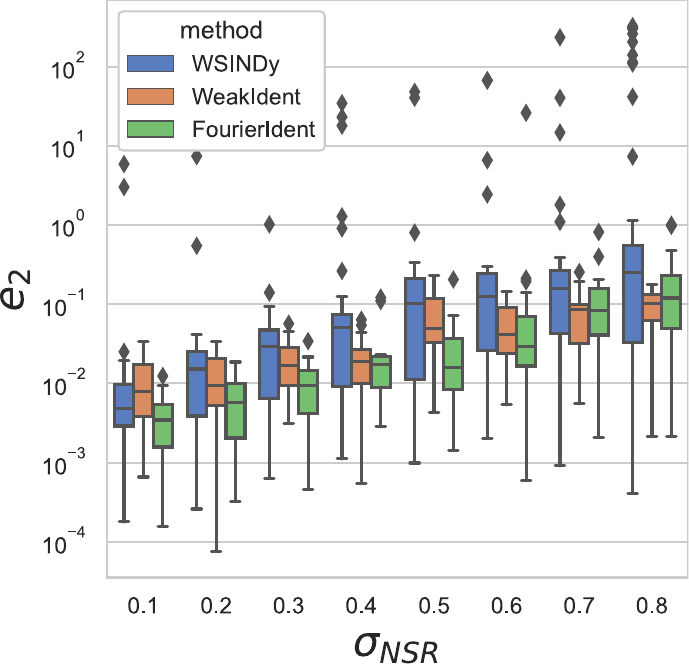}    &  \includegraphics[width = 0.23\textwidth]{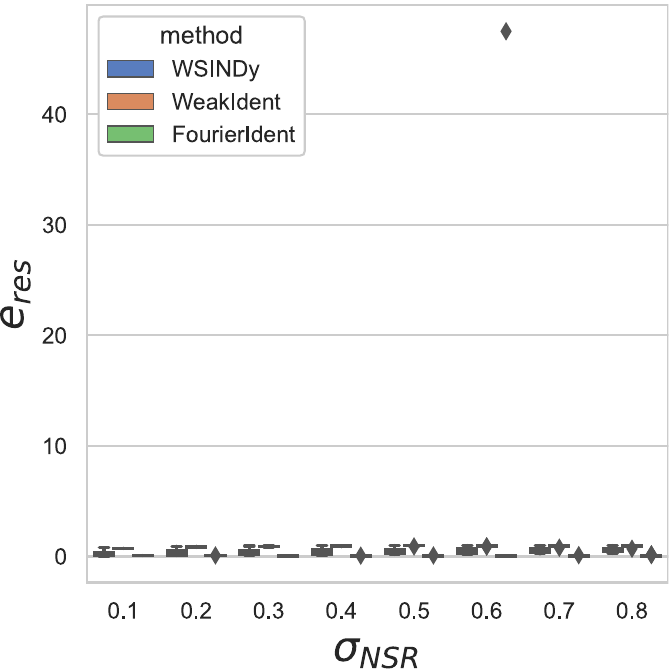}& 
    \includegraphics[width = 0.23\textwidth]{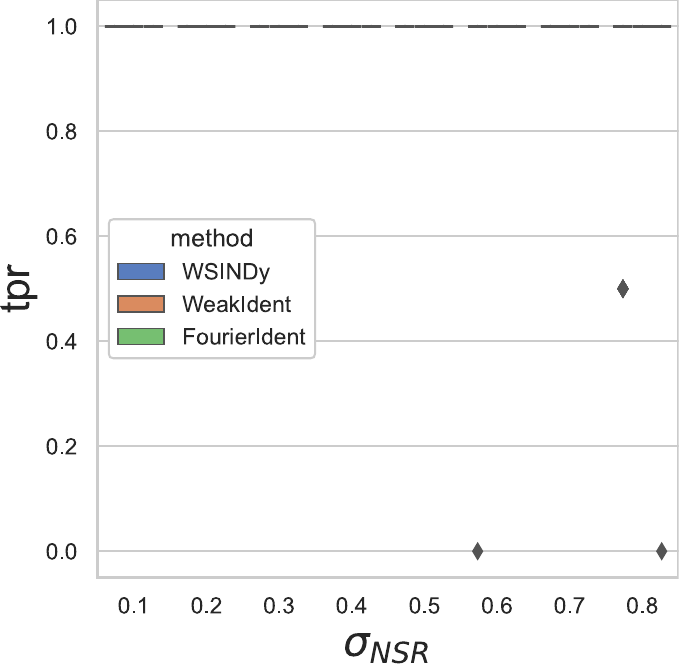}     & 
    \includegraphics[width = 0.23\textwidth]{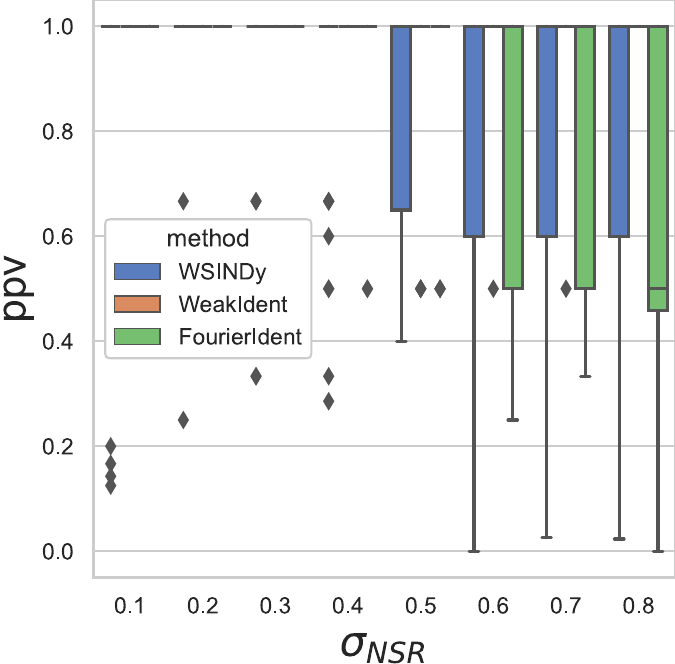}
    \end{tabular}
    \caption{Heat equation \eqref{e: heat}. (a) $e_2$ error, (b) $e_{res}$, (c) TPR, and (d) PPV. For each noise level $\sigma_{\rm NSR}$, we generate noise using 20 random seeds, and show a box plot for FourierIdent, WeakIdent, and WSINDy.  
    }
    \label{f: comparison heat} 
\end{figure}

\begin{figure}[t!]
    \centering
    \begin{tabular}{cccc}
     (a) & (b)  &      (c) & (d) \\
     \includegraphics[width = 0.23\textwidth]{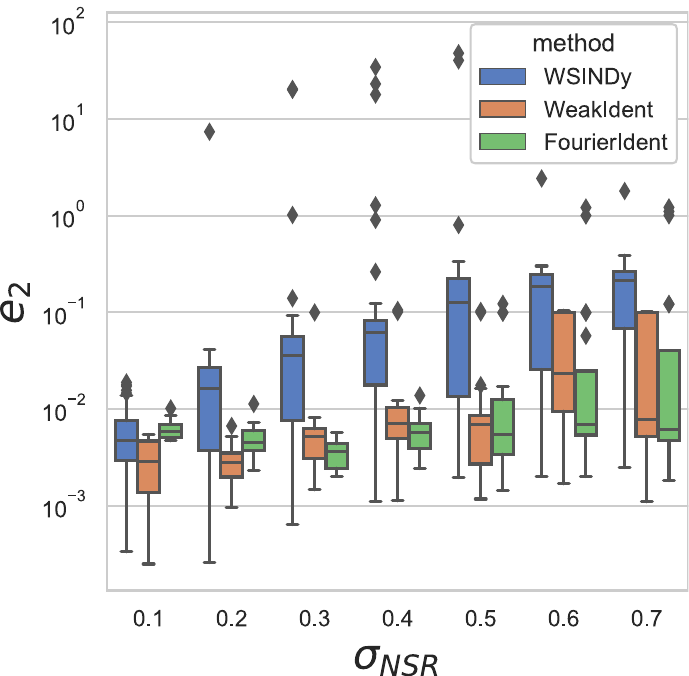}    &  \includegraphics[width = 0.23\textwidth]{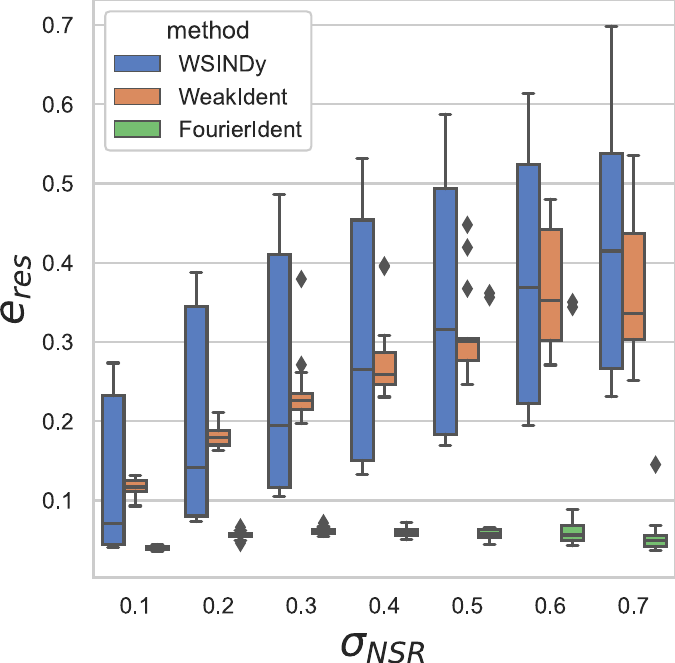} & 
    \includegraphics[width = 0.23\textwidth]{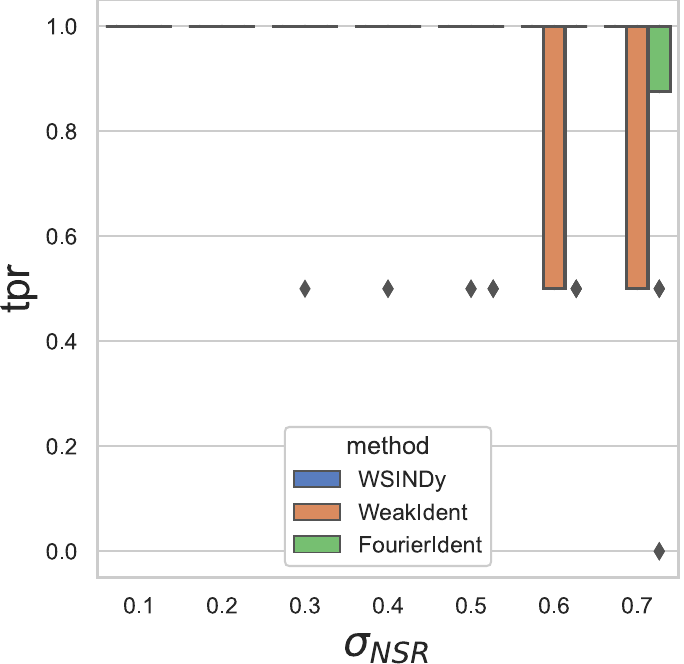}     & 
    \includegraphics[width = 0.23\textwidth]{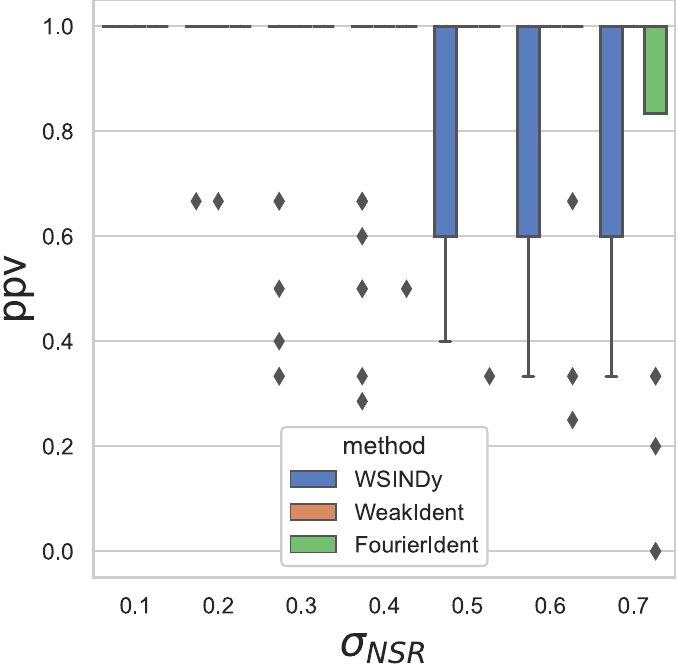}
    \end{tabular}
    \caption{Transport equation \eqref{e: tran diff}. (a) $e_2$ error, (b) $e_{res}$, (c) TPR, and (d) PPV. For each noise level $\sigma_{\rm NSR}$, we generate noise using 20 random seeds, and show a box plot for FourierIdent, WeakIdent, and WSINDy. }
    \label{f: comparison trans}
\end{figure}

\begin{figure}[t!]
    \centering
    \begin{tabular}{cccc}
     (a) & (b)  &  (c) & (d) \\
     \includegraphics[width = 0.23\textwidth]{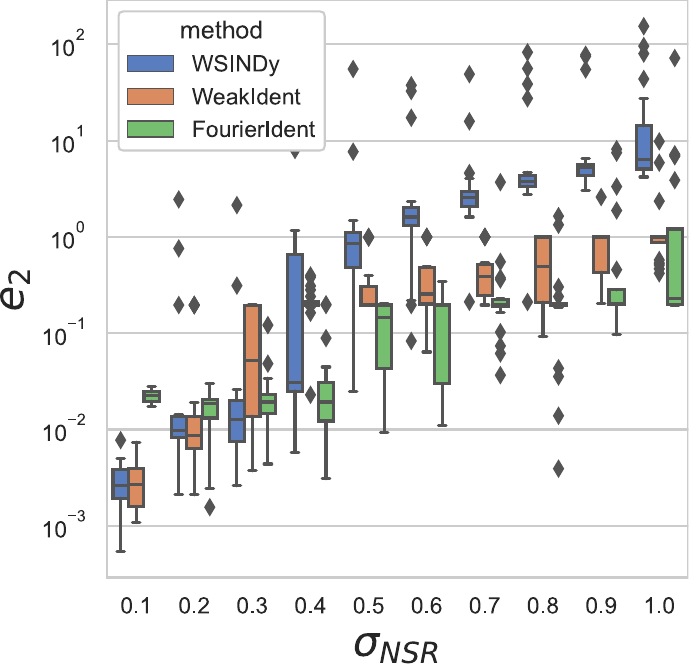}    &  \includegraphics[width = 0.23\textwidth]{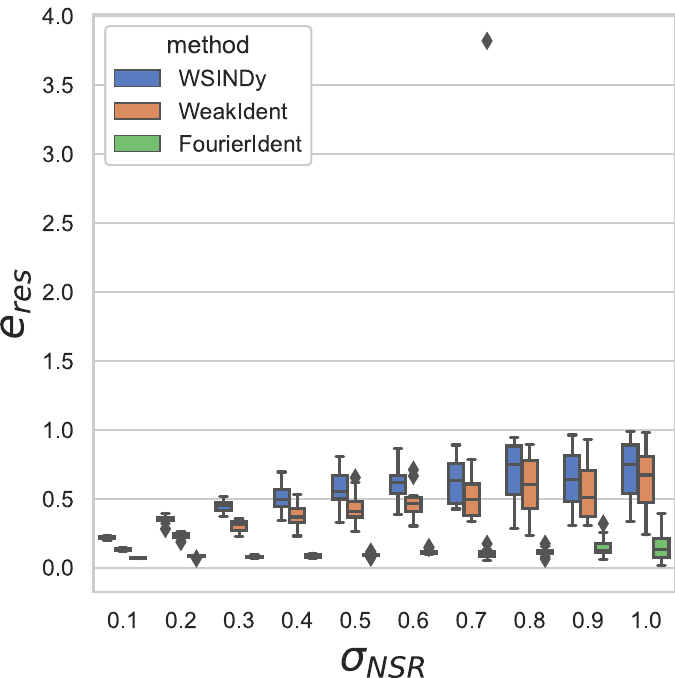}&
    \includegraphics[width = 0.23\textwidth]{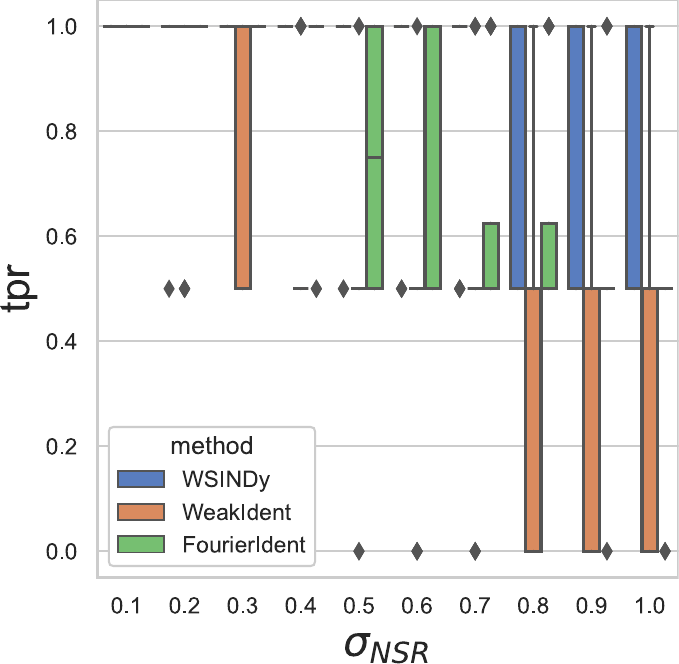}     & 
    \includegraphics[width = 0.23\textwidth]{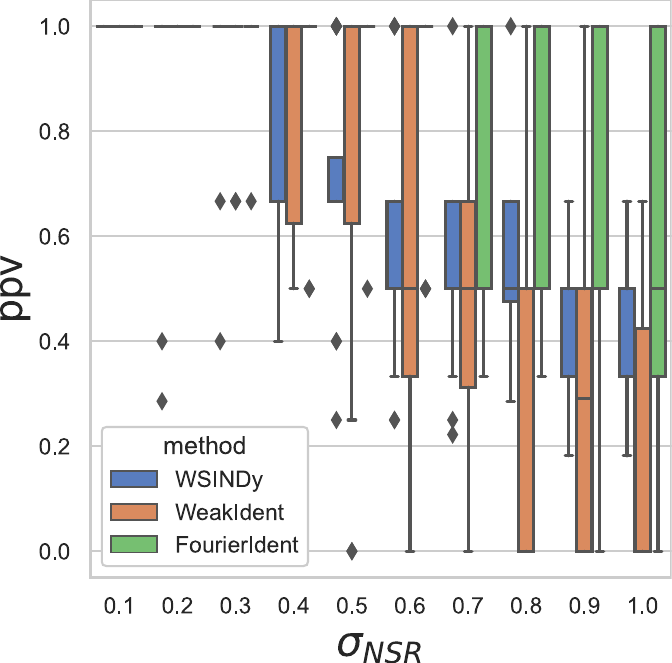}
    \end{tabular}
    \caption{The Burgers' equation \eqref{e: burgers diff w diff}.  (a) $e_2$ error, (b) $e_{res}$, (c) TPR, and (d) PPV. For each noise level $\sigma_{\rm NSR}$, we generate noise using 20 random seeds, and show a box plot for FourierIdent, WeakIdent, and WSINDy.}
    \label{f: comparison bg} 
\end{figure}

\begin{figure}[t!]
    \centering
    \begin{tabular}{cccc}
     (a) & (b)  & (c) & (d) \\
     \includegraphics[width = 0.23\textwidth]{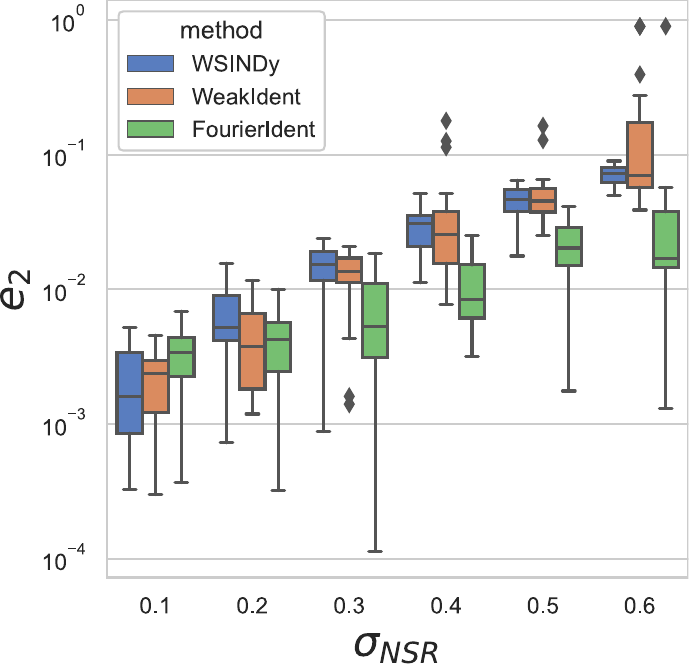}    &  \includegraphics[width = 0.23\textwidth]{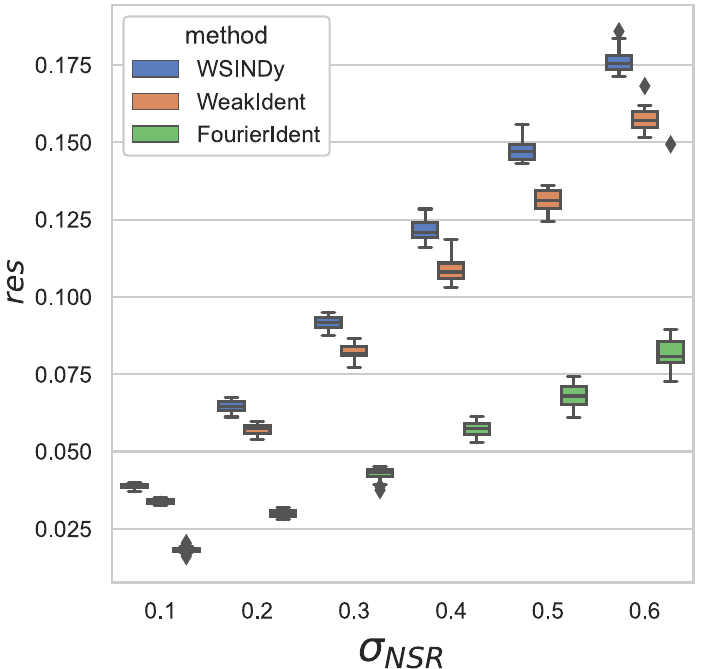}& 
    \includegraphics[width = 0.23\textwidth]{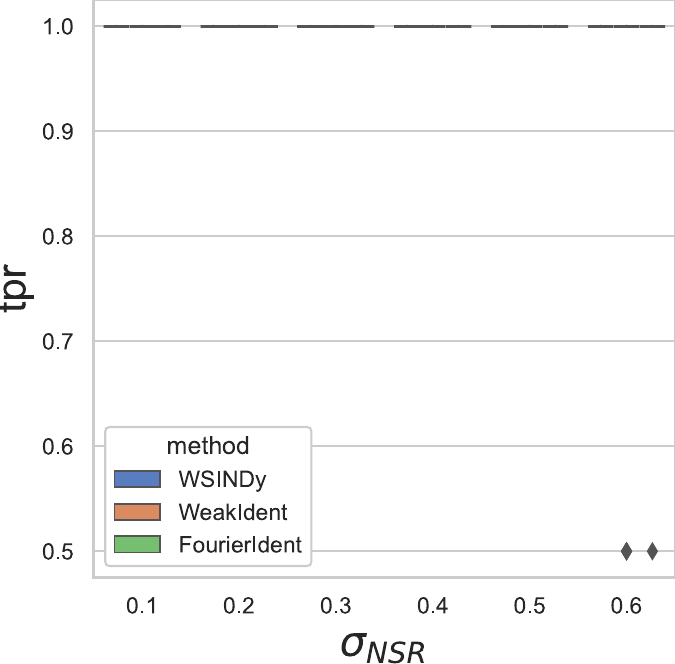}     & 
    \includegraphics[width = 0.23\textwidth]{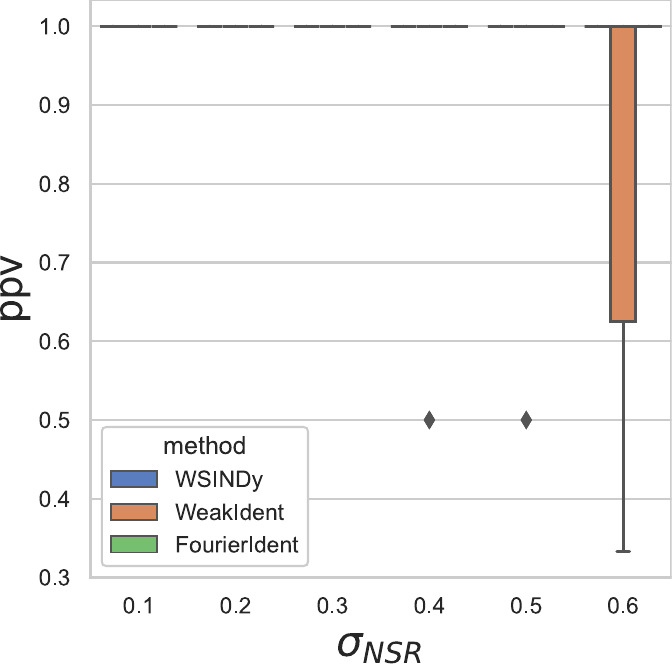}
    \end{tabular}
    \caption{The KdV equation \eqref{e: kdv}. (a) $e_2$ error, (b) $e_{res}$, (c) TPR, and (d) PPV. For each noise level $\sigma_{NSR}$, we generate noise using 20 random seeds, and show a box plot for FourierIdent, WeakIdent, and WSINDy. }
    \label{f: comparison kdv}
\end{figure}

\begin{figure}[t!]
    \centering
    \begin{tabular}{cccc}
     (a) & (b)  &      (c) & (d) \\
     \includegraphics[width = 0.23\textwidth]{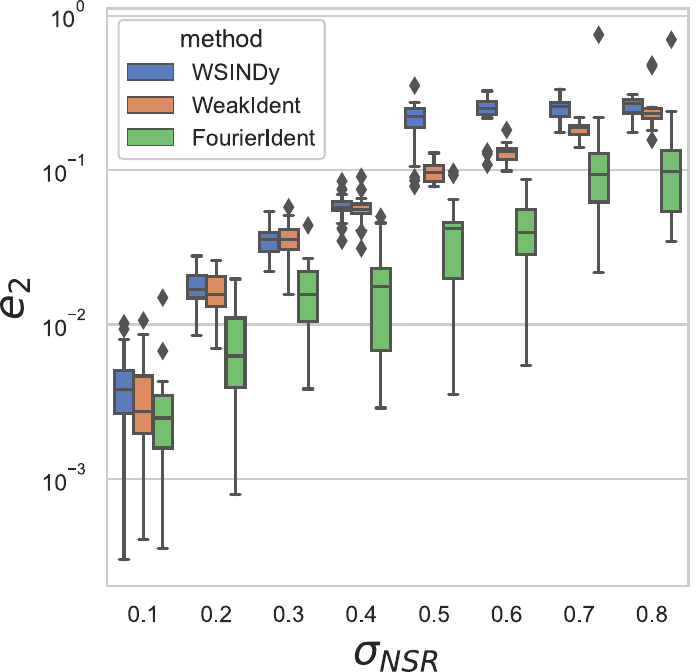}    &  \includegraphics[width = 0.23\textwidth]{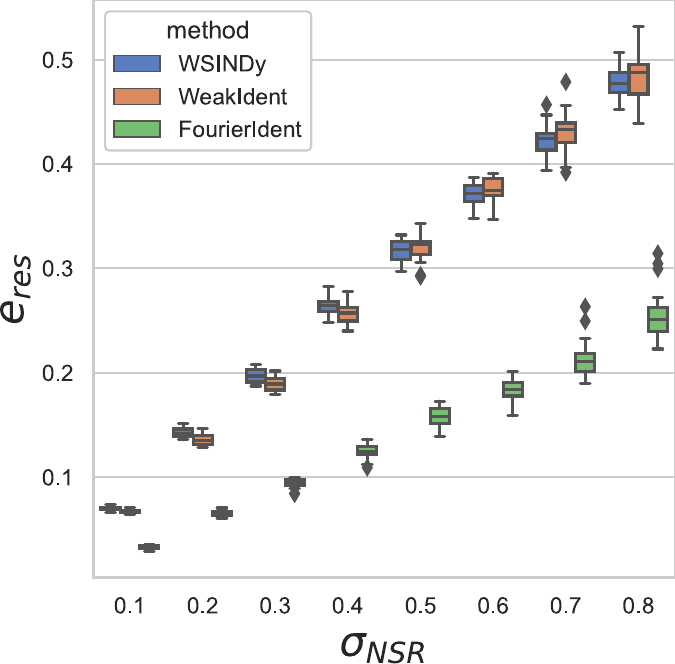}& 
    \includegraphics[width = 0.23\textwidth]{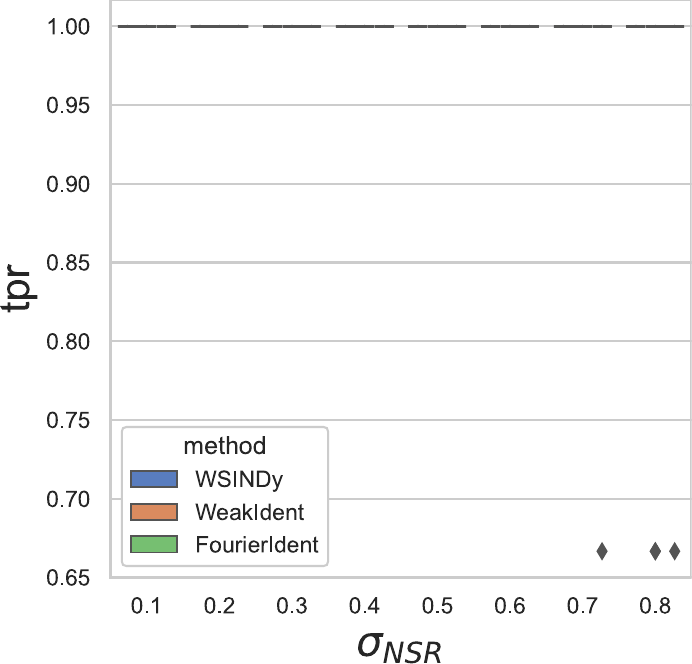}     & 
    \includegraphics[width = 0.23\textwidth]{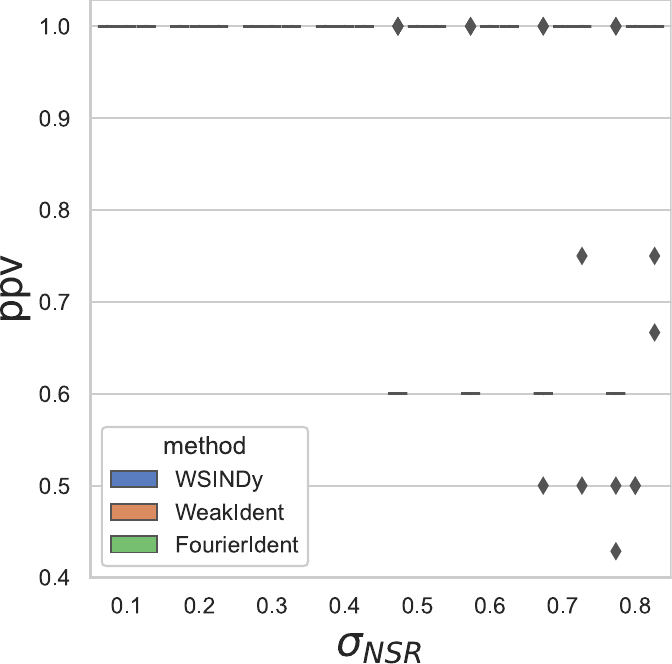}
    \end{tabular}
    \caption{Comparison results on the KS equation \eqref{e: ks} using multiple identification errors $e_2$ in (a), $e_{res}$ in (b), TPR in (c), and PPV in (d). For each noise level $\sigma_{NSR}$, we generate noise using 20 random seeds and show a box plot for FourierIdent, WeakIdent, and WSINDy. }
    \label{f: comparison ks}
\end{figure}

\section{Discussion and Conclusion} \label{s:conclusion}

We proposed FourierIdent, a method to identify differential equations in the frequency domain. We introduced denoising  Fourier features by smoothing, and frequency domain partitioning, e.g., the meaningful data region and the core regions of features.  We proposed an energy based on the core regions of features for coefficient identification.    
Different collections of high frequency responses are used to identify features and improve coefficient recovery. 
We present numerical experiments on various simulated datasets to compare FourierIdent with other state-of-art methods using weak formulations.  FourierIdent is robust to noise and can handle higher-order derivatives.  We show the benefits of FourierIdent on complex  datasets  simulated using different numbers of Fourier modes. 

FourierIdent may be computationally expensive since it iteratively finds different  collections  of active features and solves the least square problem many times,
but this part can be parallelized. 
In this paper, we consider features in the form of the derivatives of monomials.  
Expanding the feature dictionary is a possible direction in our future work.

\bibliographystyle{abbrv}
\bibliography{ref}

\end{document}